\documentclass[a4paper,12pt]{article}

\usepackage[english]{babel} % US
\usepackage{amsmath,amsfonts,amssymb,amsthm,bbm,graphics,epsfig,psfrag,epstopdf}
\usepackage[active]{srcltx}
\usepackage{paralist,subfigure,float}
\usepackage{color,soul}
\usepackage[dvipsnames]{xcolor}
\usepackage{hyperref} 
\setstcolor{red}

\newtheorem{theorem}{Theorem}[section]

\newtheorem{lemma}[theorem]{Lemma}
\newtheorem{proposition}[theorem]{Proposition}
\newtheorem{definition}[theorem]{Definition}
\newtheorem{remark}[theorem]{Remark}

\newtheorem{hypothesis}[theorem]{Hypothesis}
\newtheorem{question}[theorem]{Question}

\def\N{\mathbb{N}}
\def\Z{\mathbb{Z}}

\def\R{\mathbb{R}}
\def\epsilon{\varepsilon}
\let\e=\varepsilon
\let\vp=\varphi

\let\t=\widetilde
\let\ol=\overline
\let\ul=\underline
\let\.=\cdot
\let\0=\emptyset
\let\mc=\mathcal

\def\W{\mc{W}}
\def\hat{\widehat}
\def\tilde{\widetilde}

\DeclareMathOperator{\dist}{dist}

\DeclareMathOperator\supp{supp}

\newcommand{\su}[2]{\genfrac{}{}{0pt}{}{#1}{#2}}

\def\seq#1{(#1_n)_{n\in\N}}

\def\as{\quad\text{as }\;}

\def\for{\quad\text{for }\;}

\def\tilde{\widetilde}
\def\1{\mathbbm{1}}
\def\Sph{\mathbb{S}^{N-1}}

\newenvironment{formula}[1]{\begin{equation}\label{#1}}{\end{equation}\noindent}
\def\Fi#1{\begin{formula}{#1}}
\def\Ff{\end{formula}\noindent}

\newcommand{\be}{\begin{equation}}
\newcommand{\ee}{\end{equation}}
\newcommand{\baa}{\begin{array}}
\newcommand{\eaa}{\end{array}}
\newcommand{\ba}{\begin{eqnarray}}
\newcommand{\ea}{\end{eqnarray}}

\numberwithin{equation}{section}

\begin{document}
\date{}
\title{\bf{Spreading speeds and spreading sets of~reaction-diffusion equations}}
\author{Fran\c cois Hamel$^{\hbox{\small{ a}}}$ and Luca Rossi$^{\hbox{\small{ b,c }}}$\thanks{This work has received funding from Excellence Initiative of Aix-Marseille Universit\'e~-~A*MIDEX, a French ``Investissements d'Avenir'' programme, and from the French ANR RESISTE (ANR-18-CE45-0019) project. The first author acknowledges support of the Institut Henri Poincaré (UAR 839 CNRS-Sorbonne Universit\'e), and LabEx CARMIN (ANR-10-LABX-59-01). The first author is grateful to the hospitality of Universit\`a degli Studi di Roma La Sapienza, where part of this work was done.}\\
\\
\footnotesize{$^{\hbox{a }}$Aix Marseille Univ, CNRS, I2M, Marseille, France}\\
\footnotesize{$^{\hbox{b }}$SAPIENZA Univ Roma, Istituto ``G.~Castelnuovo'', Roma, Italy}\\
\footnotesize{$^{\hbox{c }}$CNRS, EHESS, CAMS, Paris, France}\\
}
\maketitle
	
\vspace{-5pt}

\begin{abstract}
\noindent{}This paper deals with the large time dynamics of bounded solutions of reaction-diffusion equations with unbounded initial support in $\R^N$. We prove a variational formula for the spreading speeds in any direction, and we also describe the asymptotic shape of the level sets of the solutions at large time. The Freidlin-G\"artner type formula for the spreading speeds involves newly introduced notions of bounded and unbounded directions of the initial support. The results hold for a large class of reaction terms and for solutions emanating from initial conditions with general unbounded support, whereas most of earlier results were concerned with more specific reactions and compactly supported or almost-planar initial conditions. We also prove some results of independent interest on some conditions guaranteeing the spreading of solutions with large initial support and the link between these conditions and the existence of traveling fronts with positive speed. The proofs use a mix of ODE and PDE methods, as well as some geometric arguments. The results are sharp and counterexamples are shown when the assumptions are not all fulfilled.
 
\medskip
 
\end{abstract}
	
%\tableofcontents
	
%-----------------------------------------------------------------------------------
%-----------------------------------------------------------------------------------
		
\section{Introduction and preliminaries}\label{sec:intro}
	
We are interested in the large time dynamics of solutions of the reaction-diffusion equation
\Fi{homo}
\partial_t u=\Delta u+f(u),\quad t>0,\ x\in\R^N,
\Ff
with $N\ge1$ and initial conditions $u_0$ having unbounded support. More precisely, the reaction term~$f:[0,1]\to\R$ is of class $C^1([0,1])$ with $f(0)=f(1)=0$, and the initial conditions $u_0$ are assumed to be characteristic functions $\1_U$ of sets~$U$, i.e.
\be\label{defu0}
u_0(x)=\left\{\baa{ll}1&\hbox{if }x\in U,\vspace{3pt}\\ 0&\hbox{if }x\in\R^N\!\setminus\!U,\eaa\right.
\ee
where the initial support $U$ is an unbounded measurable subset of $\R^N$ (although some results also cover the case of bounded sets $U$).\footnote{We use the term ``initial support $U$", with an abuse of notation, to refer to the set where $u_0>0$. This set differs in general from the support $\supp u_0$ of $u_0$, which is defined as the complement of the largest open set of $\R^N$ where $u_0$ is equal to $0$ almost everywhere with respect to the Lebesgue measure. However,~$U$ coincides with $\supp u_0$ if and only if $U$ is closed and the intersection of $U$ with any non-trivial ball centered at any point of $U$ has a positive Lebesgue measure.} The Cauchy problem~\eqref{homo}-\eqref{defu0} is well posed and, for each~$u_0$, there is a unique bounded classical solution $u$ of~\eqref{homo} such that~$u(t,\cdot)\to u_0$ as~$t\to0^+$ in~$L^1_{loc}(\R^N)$. For mathematical convenience, we extend~$f$ by~$0$ in~$\R\setminus[0,1]$, and the extended function, still denoted $f$, is then Lipschitz continuous in $\R$.

Instead of initial conditions $u_0=\1_U$, we can also consider multiples~$\alpha\1_U$ of charac\-teristic functions, with $\alpha>0$, at the expense of some further assumptions on the reaction term~$f$, or even other more general initial conditions $0\leq u_0\leq 1$ for which the upper level set $\{x\in\R^N:u_0(x)\geq\theta\}$ is at bounded Hausdorff distance from the support~$\supp u_0$ of~$u_0$, where $\theta\in(0,1)$ is a suitable value depending on $f$, precisely given by Proposition~\ref{pro1} below. We also refer to Remark~\ref{remtheta} below for more details. But we preferred to keep the assumption~$u_0=\1_U$ for the sake of simplicity of the presentation and of readability of the statements, all the more as this case already gives rise to many interesting and non-trivial features, depending on the type and shape of the set~$U$. 

%-----------------------------------------------------------------------------------

\subsubsection*{The main question}

Due to diffusion, the solution $u$ of~\eqref{homo}-\eqref{defu0} is of class $C^1$ in $t$ and $C^2$ in $x$ in $(0,+\infty)\times\R^N$, and
$$0<u<1\ \hbox{ in }(0,+\infty)\times\R^N$$
from the strong parabolic maximum principle, provided the Lebesgue measures of $U$ and~$\R^N\setminus U$ are positive. However, from parabolic estimates, at each finite time $t$, $u$ stays close to $1$ or $0$ in subregions of $U$ or $\R^N\setminus U$ which are far away from $\partial U$. 

One of the objectives of the present work is to describe more precisely the location at large time of the regions where~$u$ stays close to $1$ or $0$. How do these regions move and possibly spread in any direction? A fundamental issue is to understand whether and how the solution keeps a memory at large time of its initial support~$U$. A basic question is the following:

\begin{question}\label{questionA}
For a solution $u$ of~\eqref{homo}-\eqref{defu0} and for a given vector $e\in\R^N$ with unit Euclidean norm, we investigate the existence of a {\rm spreading speed}~$w(e)$ such that
\be\label{ass0}\left\{\baa{lll}
u(t,cte)\to1 & \hbox{as $t\to+\infty$} & \hbox{for every $0\le c<w(e)$},\vspace{5pt}\\
u(t,cte)\to0 & \hbox{as $t\to+\infty$} & \hbox{for every $c>w(e)$}.\eaa\right.
\ee
Can one find a formula for $w(e)$ and how does $w(e)$ depend on~$e$ and the initial~support~$U$? Moreover, is there a uniformity with respect to~$e$ in~\eqref{ass0} and, more generally speaking, are there {\rm spreading sets} which describe the asymptotic global shape of the level sets of~$u$ as~$t\to+\infty$?
\end{question}

We will answer this question in the main results, namely Theorems~\ref{th1}-\ref{th4}. The speed~$w(e)$ in~\eqref{ass0} can possibly be $+\infty$ in some directions $e$, and this actually occurs in the directions around which $U$ is unbounded, in a sense that will be made precise in Section~\ref{sec:FG}. The results will be established under some geometric assumptions on the set~$U$ and under a standard assumption on the reaction~$f$. We also provide several counterexamples in Section~\ref{sec6}. Local asymptotic symmetry properties of the level sets of~$u$ at large time are obtained in~\cite{HR2} for some reactions $f$, and further results on the flattening and more precise estimates on the location of the level sets at large time are proved in~\cite{HR3} when the initial support $U$ is a subgraph.

The situation considered in this paper can be viewed as a  counterpart of many papers devoted to the large time dynamics of solutions of~\eqref{homo} with initial conditions~$u_0$ that are compactly supported or converge to $0$ at infinity. We refer to e.g.~\cite{AW,DM1,LK,MZ1,MZ2,Z} for extinction/invasion results in terms of the size and/or the amplitude of the initial condition~$u_0$ for various reaction terms $f$, and to~\cite{DM1,DP,MP1,MP2,P1} for general local convergence and quasiconvergence results at large time. For the invading solutions~$u$ (that is, those converging to $1$ locally uniformly in $\R^N$ as $t\to+\infty$) with localized initial conditions, further estimates on the location and shape at large time of the level sets have been established in~\cite{D,G1,J,RRR,R2,R,U}. Lastly, equations of the type~\eqref{homo} set in unbounded domains~$\Omega$ instead of~$\R^N$ and notions of spreading speeds and persistence/invasion in such domains, still with compactly supported initial conditions, have been investigated in~\cite{BHN2,R3}.

The case of general unbounded initial supports $U$ has been much less investigated in the lite\-rature. One immediately sees that, for general unbounded sets $U$, Question~\ref{questionA} is much more intricate than in the case of bounded sets $U$, since the solutions $u$ can spread from all regions of the initial support $U$, that is, not only from a single bounded region. The sets~$U$ themselves can be bounded in some directions and unbounded in others. Among other things, new notions of bounded and unbounded directions of the initial support will be defined in this paper, in order to show the existence and a new variational characterization of the spreading sets of the solutions. Subtle geometric arguments will be used and new retracting super-solutions will also be introduced in the proofs. Counterexamples will also be listed when at least one of the main assumptions is not satisfied, thus showing the sharpness of the results.

%-----------------------------------------------------------------------------------

\subsubsection*{The main hypothesis and a preliminary result}\label{sec:hyp}

In this section, we list some notations and we state the main hypothesis used in the main results, as well as an important preliminary result which is a consequence of the main hypothesis. Throughout the paper, ``$|\ |$'' and ``$\ \cdot\ $'' denote respectively the Euclidean norm and inner product in $\R^N$,
$$B_r(x):=\{y\in\R^N:|y-x|<r\}$$
is the open Euclidean ball of center $x\in\R^N$ and radius $r>0$, $B_r:=B_r(0)$, and $\Sph:=\{e\in\R^N:|e|=1\}$ is the unit Euclidean sphere of~$\R^N$. The distance of a point $x\in\R^N$ to a set $A\subset\R^N$ is given by $\dist(x,A):=\inf\big\{|y-x|:y\in A\big\}$, with the convention $\dist(x,\emptyset)=+\infty$.

Since both $0$ and $1$ are steady states, one cannot determine a priori which one, if any, will win, in the sense that it will attract the solutions~$u$ of~\eqref{homo}-\eqref{defu0}. A standard way to differentiate the roles of~$0$ and~$1$ is to assume that there is a planar traveling front connecting the steady states~$1$ and~$0$:
	
\begin{hypothesis}\label{hyp}
Equation~\eqref{homo} admits a traveling front connecting $1$ to~$0$ with positive speed $c_0>0$, 
that is, a solution $u(t,x)=\vp(x\.e-c_0t)$ with $e\in\Sph$, $c_0>0$ and
\be\label{limitsvp}
1=\vp(-\infty)>\vp(z)>\vp(+\infty)=0\hbox{ \quad for all $z\in\R$}.
\ee
\end{hypothesis}

We point out that Hypothesis~\ref{hyp} only depends on the reaction term $f$, and not on the spatial dimension $N$ (hence it could actually be required in dimension $N=1$). In dimension $N>1$, the level sets of a traveling front are hyperplanes. The profile $\vp$ is necessarily decreasing (see Lemma~\ref{lem:phi'<0} below for further details). 
 
Hypothesis~\ref{hyp} is fulfilled for instance if
\be\label{positive}
f>0\ \hbox{ in $(0,1)$},
\ee
or if $f$ is of the ignition type:
\be\label{ignition}
\exists\,\alpha\in(0,1),\ \ f=0\hbox{ in $[0,\alpha]\ $ and $\ f>0$ in $(\alpha,1)$},
\ee
or if $f$ is of the bistable type:
\be\label{bistable}
\exists\,\alpha\in(0,1),\ \ f<0\hbox{ in }(0,\alpha)\ \hbox{ and }\ f>0\hbox{ in }(\alpha,1)
\ee
with $\int_0^1f(s)\,ds>0$ (in the last two cases, the speed $c_0$ is unique), see~\cite{AW,FM,F,KPP}. Hypothesis~\ref{hyp} is also satisfied for some functions~$f$ having multiple oscillations in the interval~$[0,1]$ (see~\cite{FM} and the comments on the example~\eqref{tristable} at the end of this section).

The following result, which is a preliminary result for the main results, and which has its own independent interest, shows that Hypothesis~\ref{hyp} is equivalent to the existence of a positive minimal speed $c^*$ of traveling fronts connecting $1$ to $0$, and that it implies the spreading of the solutions of~\eqref{homo} with sufficiently large initial supports.

\begin{proposition}\label{pro1}
Assume Hypothesis~$\ref{hyp}$. Then equation~\eqref{homo} admits a traveling front connecting $1$ to $0$ with minimal speed $c^*$, and $c^*>0$.\footnote{The minimality of $c^*$ means that~\eqref{homo} in $\R$ admits a solution of the form $\vp(x-c^*t)$ satisfying~\eqref{limitsvp}, and it does not admit a solution of the same type with $c^*$ replaced by a smaller quantity. Notice that, necessarily, $c^*\le c_0$ under the notation of Hypothesis~\ref{hyp}.} Furthermore, there exist $\theta\in(0,1)$ and $\rho>0$ such that if
\be\label{hyptheta}
\theta\,\1_{B_\rho(x_0)}\le u_0\le1\hbox{ in }\R^N
\ee
for some $x_0\in\R^N$, then the solution $u$ of~\eqref{homo} with initial condition $u_0$ satisfies $u(t,x)\to1$ as $t\to+\infty$ locally uniformly with respect to $x\in\R^N$, and even
\be\label{c<c*}
\forall\,c\in(0,c^*),\quad
\min_{\overline{B_{ct}}}u(t,\cdot)\to1\as t\to+\infty.
\ee
Lastly, for any compactly supported initial condition $0\le u_0\le 1$, the solution $u$ of~\eqref{homo} satisfies
\Fi{c>c*}\forall\,c>c^*,\quad
\max_{\R^N\setminus B_{ct}}u(t,\cdot)\to0\as t\to+\infty.
\Ff
\end{proposition}

Several comments are in order. Firstly, under Hypothesis~\ref{hyp}, Proposition~\ref{pro1} answers Question~\ref{questionA} for compactly supported initial data $u_0$ satisfying~\eqref{hyptheta}: namely, the solutions~$u$ then have a spreading speed $w(e)$ in any direction $e\in\Sph$, and $w(e)=c^*$. Proposition~\ref{pro1} can then be viewed as a natural extension, for instance to multistable equations such as \eqref{tristable} below, of some results of the seminal paper~\cite{AW}, which were obtained under more specific conditions on~$f$. 

Secondly, if $f$ is such that
\Fi{HTconditions}
f>0 \text{ \ in $(0,1)$\quad and \quad }\liminf_{s\to0^+}\frac{f(s)}{s^{1+2/N}}>0,
\Ff 
then Proposition~\ref{pro1} can be improved using~\cite{AW} by allowing $\theta\in(0,1)$ and $\rho>0$ in~\eqref{hyptheta} to be arbitrary. The conclusion in such a case yields the so-called {\em hair trigger effect}. If $f>0$ in~$(0,1)$ (without any further assumption on the behavior of $f(s)$ as $s\to0^+$), then, for any $\theta\in(0,1)$, there is $\rho>0$ large enough such that the conclusions of Proposition~\ref{pro1} hold, see~\cite{AW} (this fact can also be viewed as a particular case of Proposition~\ref{pro:hyp1} below). If~$f$ is of the ignition type~\eqref{ignition}, or the bistable type~\eqref{bistable} with $\int_0^1f(s)ds>0$, then $\theta$ can be any real number in the interval $(\alpha,1)$, provided $\rho>0$ is large enough, see~\cite{AW,FM}. On the other hand, without condition~\eqref{hyptheta}, the conclusion~\eqref{c<c*} of Proposition~\ref{pro1} does not hold in general, even with Hypothesis~\ref{hyp} (for instance in the bistable case~\eqref{bistable} with $\int_0^1f(s)ds>0$).

Thirdly, if there is~$\delta>0$ such that $f$ is nonincreasing in~$[0,\delta]$ and in~$[1-\delta,1]$, and if~\eqref{homo} in $\R$ has a traveling front with positive minimal speed $c^*$, then $\theta$ and $\rho$ exist as in Proposition~\ref{pro1} and~\eqref{c<c*}-\eqref{c>c*} hold, see~\cite[Theorem~1.5]{DR}. Proposition~\ref{pro1} means that this implication holds without any further assumption on $f$, and that the existence of a traveling front with positive speed is actually sufficient to get the conclusion.

Lastly, we point out that, even if there exist $\theta\in(0,1)$ and $\rho>0$ such that the initial conditions $u_0$ satisfying~\eqref{hyptheta} give rise to solutions $u$ converging to $1$ as $t\to+\infty$ locally uniformly in $\R^N$, that nevertheless does not mean in general that Hypothesis~\ref{hyp} holds. For instance, consider equation~\eqref{homo} in dimension $N=1$ with a tristable function $f$ satisfying
\be\label{tristable}\baa{l}
\exists\,0<\alpha<\beta<\gamma<1,\vspace{3pt}\\
\qquad\qquad\left\{\baa{ll}
f<0\hbox{ in }(0,\alpha),\ f>0\hbox{ in }(\alpha,\beta), & f<0\hbox{ in }(\beta,\gamma),\ f>0\hbox{ in }(\gamma,1),\vspace{3pt}\\
\displaystyle\int_0^\beta f(s)ds>0, & \displaystyle\int_\beta^1 f(s)ds>0,\eaa\right.\eaa
\ee
and let $c_1$ and $c_2$ be the unique (positive) speeds of the traveling fronts $\vp_1(x-c_1t)$ and $\vp_2(x-c_2t)$ connecting $\beta$~to~$0$, and~$1$ to~$\beta$, respectively. On the one hand, if~$c_1\ge c_2$, then Hypothesis~\ref{hyp} is not satisfied (in such a case instead of a single front there exists a terrace, see~\cite{FM}) but if $u_0$ fulfills~\eqref{hyptheta} with given $\theta\in(\gamma,1)$ and a sufficiently large $\rho>0$, then $u$ satisfies $u(t,\cdot)\to1$ as $t\to+\infty$ locally uniformly in $\R^N$, and even~\eqref{c<c*} for some $c^*>0$, see~\cite{FM} (or Proposition~\ref{pro:hyp1} and Remark~\ref{rem:spreadingspeed} below). On the other hand, Hypothesis~\ref{hyp} is satisfied if (hence, only if) $c_1<c_2$, see again~\cite{FM}.

%-----------------------------------------------------------------------------------
%-----------------------------------------------------------------------------------
	
\section{Main results}\label{sec:FG}

\subsubsection*{The main results: Freidlin-G\"artner type formula, spreading speeds, spreading sets}

In this section, under Hypothesis~\ref{hyp}, we investigate the asymptotic set of spreading for the solutions $u$ of~\eqref{homo}-\eqref{defu0} with general unbounded initial supports $U$ containing large enough balls. Such solutions $u$ then converge to $1$ as $t\to+\infty$ locally uniformly in $\R^N$, and even satisfy~\eqref{c<c*}, with $c^*>0$ given by Proposition~\ref{pro1}. But we now want to provide a more precise description of the invasion of the state~$0$ by the state $1$. We point out that the invasion cannot be uniform in all directions in general, since it shall strongly depend on the initial support~$U$. For~$e\in\Sph$, we then look for a quantity $w(e)\in(0,+\infty]$ satisfying~\eqref{ass0}. This quantity would then be referred to as the {\em spreading speed} and it would represent the asymptotic speed at which the level sets, with levels between $0$ and $1$, move in the direction $e$. If it exists, it necessarily satisfies $w(e)\ge c^*$ by Proposition~\ref{pro1}. However, in contradistinction with the case of compactly supported initial conditions, the spreading speeds of solutions of~\eqref{homo}-\eqref{defu0} may not exist when the initial supports~$U$ are unbounded (see Proposition~\ref{pro:ce1} below).

In order to give the existence of and a formula for the spreading speeds in all directions, we introduce the notions of sets of directions ``around which~$U$ is bounded'' and ``around which~$U$ is unbounded'', for short the set of {\it bounded directions} and the set of {\it unbounded directions}, defined by:
$$\mc{B}(U):=\Big\{\xi\in\Sph:\liminf_{\tau\to+\infty}\frac{\dist(\tau\xi,U)}{\tau}>0\Big\}$$
and
$$\mc{U}(U):=\Big\{\xi\in\Sph:\lim_{\tau\to+\infty}\frac{\dist(\tau\xi,U)}{\tau}=0\Big\}.$$
The sets~$\mc{B}(U)$ and $\mc{U}(U)$ are respectively open and closed relatively to $\Sph$. The condition $\xi\in\mc{B}(U)$ is equivalent to the existence of an open cone $\mc{C}$ containing the ray~$\R^+\{\xi\}=\{\tau\,\xi:\tau>0\}$ such that $U\cap\mc{C}$ is bounded, that is, $\R^+\{\xi\}\,\subset\,\mc{C}\,\subset\,(\R^N\!\setminus\!U)\cup B_R$ for some $R>0$. Conversely, for any $\xi\in\mc{U}(U)$ and any open cone $\mc{C}$ containing the ray~$\R^+\{\xi\}$, the set $U\cap\mc{C}$ is then unbounded. We also define the notion of positive-distance-interior~$U_\rho$ (with $\rho>0$) of the set $U$ as
$$U_\rho:=\big\{x\in U:\dist(x,\partial U)\ge\rho\big\}.$$

Our first main result shows the existence of spreading speeds, under a 
geometric condition on $U$ in~\eqref{defu0}, and it provides an explicit formula for these speeds.

\begin{theorem}\label{th1}
Assume that Hypothesis $\ref{hyp}$ holds, let~$c^*>0$ and $\rho>0$ be given by Proposition~$\ref{pro1}$, and let~$u$ be the solution of~\eqref{homo}-\eqref{defu0}, with $U\subset\R^N$ satisfying $U_\rho\neq\emptyset$~and
\Fi{hyp:U}
\mc{B}(U)\cup\,\mc{U}(U_\rho)=\Sph.
\Ff
Then, for every $e\in\Sph$, there exists $w(e)\in[c^*,+\infty]$ such that~\eqref{ass0} holds, as well as the stronger property
\Fi{ass}\left\{\baa{lll}
\displaystyle\lim_{t\to+\infty}\,\Big(\min_{0\le s\le c}u(t, ste)\Big)\!&\! =1 & \text{for every }0\le c<w(e),\vspace{3pt}\\
\displaystyle\lim_{t\to+\infty}\,\Big(\sup_{s\ge c}u(t,ste)\Big)\!&\!=0 & \text{for every }c>w(e).\eaa\right.
\Ff
Furthermore, $w(e)$ is given explicitly by the variational formula
\Fi{FGgeneral}
w(e)=\sup_{\xi\in\mc{U}(U),\ \xi\.e\ge0}\frac{c^*}{\sqrt{1-(\xi\.e)^2}}
\Ff
with the conventions
\be\label{conv}\left\{\baa{ll}
w(e)=c^* & \hbox{if there is no $\xi\in\mc{U}(U)$ such that $\xi\.e\ge0$},\vspace{3pt}\\
w(e)=+\infty & \hbox{if $e\in\mc{U}(U)$.}\eaa\right.
\ee
\end{theorem}

Several comments are in order, on formula~\eqref{FGgeneral} and on the hypotheses and conclusions of Theorem~\ref{th1}. Firstly, it follows from~\eqref{FGgeneral}-\eqref{conv} that $w(e)\geq c^*$ for all $e\in\Sph$, and that the inequality is strict if and only if there is $\xi\in\mc{U}(U)$ such that $\xi\cdot e>0$. We also observe that, for an arbitrary set $U$ satisfying $\mc{U}(U)\neq\emptyset$, formula~\eqref{FGgeneral} with the convention~\eqref{conv} can be rephrased in a more geometric way:
\Fi{FGgeometric}
w(e)\,=\,\frac{c^*}{\dist(e,\R^+\, \mc{U}(U))}\,=\,\frac{c^*}{\sin\vartheta}\,,
\Ff
where $\vartheta\in[0,\pi/2]$ is the minimum between~$\pi/2$ and the smallest angle between the direction~$e$ and the directions in~$\mc{U}(U)$ (with the convention $c^*/0=+\infty$). This immediately implies the continuity of the map $e\mapsto w(e)\in[c^*,+\infty]$ in $\Sph$. If $\mc{U}(U)=\emptyset$, then the map $e\mapsto w(e)$ is constantly equal to $c^*$, whence continuous in $\Sph$. 

Secondly, we call~\eqref{FGgeneral} a Freidlin-G\"artner type formula, since Freidlin and G\"artner~\cite{FG} were the first to characterize the spreading speeds of solutions of reaction-diffusion equations in~$\R^N$ by a variational formula. They were actually concerned with spreading speeds for solutions of $x$-dependent reaction-diffusion equations of the Fisher-KPP type \cite{F,KPP} (for which $0<f(x,u)/u\le\frac{\partial f}{\partial u}(x,0)$ for all $(x,u)\in\R^N\times(0,1)$) with $f(x,u)$ periodic with respect to~$x$: more precisely, it follows from~\cite{FG}, together with~\cite{BHN,BHN1,W}, that~\eqref{ass0} holds for solutions emerging from compactly supported initial data, with
\be\label{formuleFG}
w(e)=\inf_{\xi\in\Sph,\,\xi\cdot e>0}\ \frac{c^*(\xi)}{\xi\cdot e}
\ee
for any $e\in\Sph$, where $c^*(\xi)$ denotes the minimal speed of pulsating fronts connecting $1$ to $0$ in the direction $\xi$.\footnote{A pulsating front connecting $1$ to $0$ with speed $c$ in the direction $\xi$ is a solution $u:\R\times\R^N\to(0,1)$ which can be written as $u(t,x)=\phi(x\cdot\xi-ct,x)$, where $\phi(-\infty,x)=1$, $\phi(+\infty,x)=0$ uniformly in $x\in\R^N$, and $\phi$ has the same periodicity with respect to its second argument as the function $f$ or other coefficients of the equation, see e.g.~\cite{BH,SKT,W,X1,X2}.} Such formulas for the spreading speeds of solutions with compactly supported initial conditions have been recently extended to more general reaction terms in~\cite{R1}. Formula~\eqref{formuleFG} reveals that, for reaction-diffusion equations with spatially periodic coefficients, the spreading speed $w(e)$ may depend on the direction~$e$, even for solutions with compactly supported initial conditions~$u_0$. However, the continuity of the map~$e\mapsto w(e)$ still holds for monostable, ignition or bistable reactions~$f$, as follows from~\cite{FG,R1} and from the (semi-)continuity of the minimal or unique speeds of pulsating traveling fronts with respect to the direction, see~\cite{AG,G,R1} (but the continuity of the spreading speeds and even their existence do not hold in general when pulsating fronts connecting $1$ to $0$ do not exist anymore, see~\cite{GR}).

Thirdly, regarding the assumptions of Theorem~\ref{th1}, we first remind that Hypothesis~\ref{hyp} holds in the positive case~\eqref{positive}, in the ignition case~\eqref{ignition}, and in the bistable case~\eqref{bistable} with~$\int_0^1f(s)ds>0$. Hence Theorem~\ref{th1} yields the existence of the spreading speeds satisfying~\eqref{ass0} and~\eqref{ass}, given by~\eqref{FGgeneral}-\eqref{conv}, as soon as the initial datum~$u_0=\1_U$ is associated with a set~$U\subset\R^N$ satisfying $U_\rho\neq\emptyset$ and~\eqref{hyp:U}. Moreover, in the case of a positive nonlinearity satisfying~\eqref{HTconditions}, for which the hair trigger effect holds, it suffices that such geometric conditions on~$U$ hold with~$\rho>0$ arbitrarily small. On the other hand, we point out that the conclusions of Theorem~\ref{th1} do not hold in general without Hypothesis~\ref{hyp}: for instance, for a tristable function $f$ of the type~\eqref{tristable} with $c_1>c_2$ (where~$c_1$ and~$c_2$ are the positive speeds of the traveling fronts $\vp_1(x-c_1t)$ and $\vp_2(x-c_2t)$ connecting $\beta$~to~$0$, and~$1$ to~$\beta$, respectively), then the solutions $u$ of~\eqref{homo}-\eqref{defu0} with $U$ bounded (hence,~\eqref{hyp:U} is satisfied) and $U_\rho\neq\emptyset$, develop into a terrace of expanding fronts with speeds $c_1$ and $c_2$: $\inf_{B_{ct}}u(t,\cdot)\to1$ as $t\to+\infty$ if $0<c<c_2$ (resp. $\sup_{B_{c''t}\setminus B_{c't}}|u(t,\cdot)-\beta|\to0$ as $t\to+\infty$ if $c_2<c'<c''<c_1$, resp. $\sup_{\R^N\setminus B_{ct}}u(t,\cdot)\to0$ as $t\to+\infty$ if $c>c_1$), see~\cite{DM2,DGM,FM,P2}. Hence, the existence of $w(e)$ satisfying~\eqref{ass0} fails.

Finally, let us comment on the geometric assumption~\eqref{hyp:U}, which is readily seen to be invariant under rigid transformations of~$U$. Some sufficient conditions for the validity of~\eqref{hyp:U} are given in Proposition~\ref{pro:hypU} below. We also point out that $\mc{U}(U_\rho)$ cannot be replaced by $\mc{U}(U)$ in~\eqref{hyp:U}, see Remark~\ref{remhypU} below. Here are some sufficient conditions and examples to have that a direction $\xi$ belongs to~$\mc{U}(U_\rho)$:
\begin{itemize}
\item $U\cup B_R\supset\mc{C}$, for some $R>0$ and some open cone $\mc{C}$ containing the ray $\R^+\{\xi\}$, or, more generally, for a half-cylinder $\mc{C}$ with axis $\xi$ and whose section orthogonal to $\xi$ contains an $(N-1)$-dimensional ball of radius~$\rho$;
\item $U$ satisfies the uniform interior sphere condition of radius $\rho$ (that is, for every $p\in\partial U$, there is $a\in U$ such that $|a-p|=\rho$ and $B_{\rho}(a)\subset U$) and it is {\em strongly unbounded} in the direction $\xi$, in the sense that~$U+B_R\supset\R^+\{\xi\}$ for some $R>0$;
\item $U\cup B_R\,\supset\,\bigcup_{n\in\N} B_{\rho}(n^\alpha \xi)$ for some $R>0$ and $\alpha>0$ (observe that, when $\alpha>1$,~the distance between two consecutive centers is $|(n + 1)^\alpha\xi - n^\alpha\xi|\sim \alpha n^{\alpha-1}\!\to\!+\infty$ as~$n\!\to\!+\infty$).
\end{itemize}

Our second main result asserts the uniformity of the limits~\eqref{ass0} with respect to the directions $e\in\Sph$.

\begin{theorem}\label{th2}
Under the same assumptions and notations as in Theorem~$\ref{th1}$, it holds that, for any compact set $C\subset\R^N$,
\Fi{ass-cpt}\left\{\baa{lll}
\displaystyle\lim_{t\to+\infty}\,\Big(\min_{x\in C}u(t, tx)\Big)\!&\! =1 & \text{if }\;C\subset\mc{W},\vspace{3pt}\\
\displaystyle\lim_{t\to+\infty}\,\Big(\max_{x\in C}u(t,tx)\Big)\!&\!=0 & \text{if }\;C\subset\R^N\setminus\ol{\mc{W}},\eaa\right.
\Ff
where $\mc{W}$ is the envelop set of the function $w:\Sph\to[c^*,+\infty]$ defined by~\eqref{FGgeneral}-\eqref{conv}, i.e.
\be\label{asspre}
\W:=\big\{re\,:\,e\in \Sph,\ \ 0\leq r<w(e)\big\},
\ee
which has the following expression:
\be\label{asspre-formula}
\W=\R^+\, \mc{U}(U)\,+\,B_{c^*}
\ee
$($with the convention that \,$\R^+\emptyset\,+\,B_{c^*}=B_{c^*}$$)$.
\end{theorem}

Formula~\eqref{asspre-formula} means that the envelop set $\W$  of the speeds $w(e)$'s coincides with the $c^*$-neighborhood of the positive cone generated by the directions $\mc{U}(U)$. One immediately deduces that $\mc{W}$ is an open set which is either unbounded (when $\mc{U}(U)\neq\emptyset$), or it coincides with $B_{c^*}$. For periodic Fisher-KPP equations, formula~\eqref{formuleFG} given in~\cite{BHN,BHN1,FG,W} for the spreading speeds of solutions with compactly supported initial conditions means that~$\overline{\W}$, if~$\W$ were still defined by~\eqref{asspre}, would be the Wulff shape (whence the letter~$\W$) of the envelop set of the minimal speeds $c^*(\xi)$ of pulsating fronts and, since the map $\xi\mapsto c^*(\xi)\in(0,+\infty)$ is continuous in $\Sph$ by~\cite{AG,FG,R1}, the set $\overline{\W}$ would therefore be a convex compact set. Instead, for our pro\-blem~\eqref{homo}-\eqref{defu0} under the assumptions of Theorem~$\ref{th1}$, formula \eqref{asspre-formula} shows that the set $\W$ defined in~\eqref{asspre} is not bounded as soon as $\mc{U}(U)\neq\emptyset$, and moreover that it is not convex in general. For instance, if~$U\neq\emptyset$ is a non-convex closed cone, say with vertex $0$, then~$\R^+\mc{U}(U)\cup\{0\}=U$ and thus, from~\eqref{asspre-formula},~$\W$ is not convex either (see~\eqref{Ugamma}-\eqref{hypgamma1} below for further similar examples). Never\-theless, if $U$ is a general convex set, then $\R^+\mc{U}(U)\cup\{0\}$ is convex, and $\W$ is convex too, from~\eqref{asspre-formula} again. More generally speaking, if there is a convex set $U'$ which lies at a finite Hausdorff distance from~$U$ (see~\eqref{defdH} below for a more precise definition), then $\mc{U}(U)=\mc{U}(U')$ and therefore~$\W$ is convex, even if~$U$ itself is not.

On the basis of~\eqref{ass-cpt}, we say that $\W$ is a {\em spreading set} for~\eqref{homo}-\eqref{defu0}. We point out that property~\eqref{ass-cpt} contains~\eqref{ass0}, owing to the continuity of the map $e\mapsto w(e)$ in~$\Sph$. It also yields the first line of~\eqref{ass} by taking $C$ as the segment between~$0$ and~$ce$ with $0\le c<w(e)$. Compared to the first lines of~\eqref{ass0} and~\eqref{ass}, the first line of~\eqref{ass-cpt} provides an additional uniformity with respect to the directions~$e$. It also follows from~\eqref{ass-cpt} and the continuity of the map $e\mapsto w(e)\in[c^*,+\infty]$ that, for any $\sigma\in(0,1)$ and $A>0$,
\be\label{spreadingset}
\min_{x\,\in\,\sigma\overline{\W}\cap\overline{B_A}}u(t,tx)\to1\ \hbox{ as }t\to+\infty.
\ee
Formulas similar to~\eqref{ass-cpt} have been established for the solutions of more general heterogeneous equations or systems with compactly supported initial conditions and Fisher-KPP reactions~\cite{BES,BN,ES,LM,W} or bistable reactions~\cite{X2}. The main difference is that, in these references, the spreading speeds and sets are bounded, unlike the spreading set $\W$ defined in~\eqref{asspre}-\eqref{asspre-formula}, which is unbounded as soon as $\mc{U}(U)\neq\emptyset$ (whence the use of $A$ in~\eqref{spreadingset}, for the set $\sigma\overline{\W}\cap\overline{B_A}$ to be compact). A different method based on the scaling $(t,x)=(t'/\epsilon,x'/\epsilon)$ and the limit $\epsilon\to0$ was used in~\cite{BES,ES,X2} and also in~\cite{BBS,G2} (where initial conditions which are not compactly supported can be considered). This method can however not be used for~\eqref{homo}-\eqref{defu0} since the set $U$ is not invariant by rescaling, even up to translation, unless it is a positive cone. 

Theorems~\ref{th1}-\ref{th2} were concerned with the convergence to $1$ and $0$ as $t\to+\infty$ along some rays or some dilated sets. Our next two results Theorems~\ref{th3}-\ref{th4}, the first one following actually from Theorem~\ref{th2} (as seen from its proof in Section~\ref{sec5} below), provide a description of the asymptotic shape of the upper level sets of a solution $u$, defined for $\lambda\in(0,1)$ and $t>0$ by
\be\label{defElambda}
E_\lambda(t):=\big\{x\in\R^N:u(t,x)>\lambda\big\}.
\ee
That description involves the Hausdorff distance between some sets depending on~$E_\lambda(t)$ and~$t\W$. The Hausdorff distance is defined, for any pair of subsets $A,B\subset\R^N$, by  
\be\label{defdH}
d_{\mc{H}}(A,B):=\max\Big(\sup_{x\in A}\dist(x,B),\,\sup_{y\in B}\dist(y,A)\Big),
\ee
with the conventions that $d_{\mc{H}}(A,\emptyset)=d_{\mc{H}}(\emptyset,A)=+\infty$ if $A\neq\emptyset$ and $d_{\mc{H}}(\emptyset,\emptyset)=0$. Notice that a first relation between $E_\lambda(t)$ and $t\mc{W}$ immediately follows from the continuity of the map $e\mapsto w(e)$ in $\Sph$ and from~\eqref{ass0}, provided this formula holds: for any $\lambda\in(0,1)$,
$$\begin{cases}
x\in\mc{W}\implies tx\in E_\lambda(t) \\
x\notin\ol{\mc{W}}\implies tx\notin E_\lambda(t) 
\end{cases}
\text{for large } t.$$

\begin{theorem}\label{th3}
Under the same assumptions and notations as in Theorems~$\ref{th1}$-$\ref{th2}$, it holds that, for any $R>0$ and any $\lambda\in(0,1)$,
\Fi{Hloc}
d_{\mc{H}}\Big(\ol{B_R}\cap\frac1t\, E_\lambda(t)\,,\,\ol{B_R}\cap \mc{W}\Big)\to0\ \text{ as }t\to+\infty.
\Ff
\end{theorem}

As a matter of fact, property~\eqref{Hloc} can be extended by replacing $\ol{B_R}$ with any compact set 
$K\subset\R^N$ satisfying $\overline{K\cap\mc{W}}=K\cap\overline{\mc{W}}$, as will be seen in the proof of Theorem~\ref{th3} in Section~\ref{sec5}.

Theorem~\ref{th3} gives the approximation of $t^{-1}E_\lambda(t)$ by $\mc{W}$ {\it locally} with respect to the Hausdorff distance as $t\to+\infty$. But we point out that this convergence is not {\it global} in general, that is, the truncation by the compact set $K$ is truly needed for~\eqref{Hloc} to hold and $d_{\mc{H}}(t^{-1}E_\lambda(t),\mc{W})\not\to0$ as $t\to+\infty$ in general, even under the assumptions of Theorem~\ref{th3} (see Proposition~\ref{pro:ce3} below and also the discussion at the end of this section about the possible lack of convergence of~$t^{-1}E_\lambda(t)$). 

However, the following and last main result provides an asymptotic approximation of~$t^{-1}E_\lambda(t)$ by a {\em family of sets}, namely $t^{-1}U+B_{c^*}$, {\it globally} with respect to the Hausdorff distance. For this, we do not need the geometric assumption~\eqref{hyp:U}, but rather that the Hausdorff distance between $U$ and its positive-distance-interior $U_\rho$ is finite.

\begin{theorem}\label{th4}
Assume that Hypothesis~$\ref{hyp}$ holds, let~$c^*>0$ and $\rho>0$ be given by Proposition~$\ref{pro1}$, and let~$u$ be the solution of~\eqref{homo}-\eqref{defu0}, with $U\subset\R^N$ satisfying $U_\rho\neq\emptyset$~and
\Fi{dUrho}
d_{\mc{H}}(U,U_\rho)<+\infty.
\Ff
Then, for any $\lambda\in(0,1)$, there holds that
\Fi{dH}
d_{\mc{H}}\big( E_\lambda(t) \,,\, U+B_{c^*t}\big)=o(t)\quad\text{as }\;t\to+\infty.
\Ff
\end{theorem}

Property~\eqref{dH} means that $E_\lambda(t)$ behaves at large time $t$ as the set $U$ thickened by~$c^*t$. A~sufficient condition for~\eqref{dUrho} to hold is that the set $U$ fulfills the uniform interior sphere condition of radius~$\rho$: in such a case $d_{\mc{H}}(U,U_\rho)\leq2\rho$. In particular, if~$f$ satisfies condition~\eqref{HTconditions} ensuring the hair trigger effect, then Theorem~\ref{th4} applies to any  non-empty set $U$ which is uniformly $C^{1,1}$.  

We point out that a single formula like~\eqref{dH} valid for all $\lambda\in(0,1)$ does not hold in general without Hypo\-thesis~\ref{hyp}. For instance, consider the same counter-example as the one mentioned after Theorem~\ref{th1}, namely, take a tristable function $f$ of the type~\eqref{tristable} with $c_1>c_2$ (where~$c_1$ and~$c_2$ are the positive speeds of the traveling fronts $\vp_1(x-c_1t)$ and $\vp_2(x-c_2t)$ connecting~$\beta$~to~$0$, and~$1$ to~$\beta$, respectively). Then, as follows from~\cite{DM2,DGM,FM,P2}, the solutions~$u$ of~\eqref{homo}-\eqref{defu0} with $U$ bounded and $U_\rho\neq\emptyset$ (hence,~\eqref{dUrho} is satisfied) are such that $d_{\mc{H}}\big(E_\lambda(t),U+B_{c_2t}\big)=o(t)$ as $t\to+\infty$ if $\beta<\lambda<1$, respectively $d_{\mc{H}}\big(E_\lambda(t),U+B_{c_1t}\big)=o(t)$ as $t\to+\infty$ if $0<\lambda<\beta$ (the behavior of $E_{\lambda}(t)$ when $\lambda=\beta$ is actually unclear).

\subsubsection*{Common comments and counter-examples to Theorems~\ref{th1}-\ref{th4}}

When $U$ is bounded in~\eqref{defu0}, then $\mc{U}(U)=\emptyset$, $\mc{B}(U)=\Sph$, hence~\eqref{hyp:U} is automatically fulfilled, as is~\eqref{dUrho} provided $U_\rho\neq\emptyset$. Theorems~\ref{th1}-\ref{th4}, which imply that \eqref{ass0}, \eqref{ass}, \eqref{ass-cpt}, \eqref{Hloc} and~\eqref{dH} hold with~$w(e)\equiv c^*$ in $\Sph$ and $\mc{W}=B_{c^*}$, can then be viewed in that case as a consequence of Proposition~\ref{pro1}. 

Now, in the class of unbounded sets $U$, a particular class is that of periodic sets, or sets $U$ containing periodic sets. Namely, consider $r,R>0$ and assume that
$$U\supset\bigcup_{k\in\Z^N}B_r(Rk).$$
Under Hypothesis~\ref{hyp}, and with $\rho>0$ given by Proposition~\ref{pro1}, if $r\ge\rho$ then Proposition~\ref{pro1} and the periodicity of $u(t,\cdot)$ imply that $u(t,\cdot)\to1$ as $t\to+\infty$ uniformly in~$\R^N$ (notice also that, here, $\mc{U}(U_\rho)=\Sph$, $w(e)=+\infty$ for all $e\in\Sph$, and $\W=\R^N$). On the other hand, the uniform convergence to $1$ of the solutions $u$ of~\eqref{homo}-\eqref{defu0} does not hold in general for periodic sets $U$. Consider for instance a bistable function~$f$ of the type~\eqref{bistable} and assume that
$$U\subset\bigcup_{k\in\Z^N}B_r(Rk).$$
Then $0\le u(1,x)\le e^L\int_{B_r}G_1(x,y)dy$ for all $x\in\R^N$, with $L=\max_{[0,1]}|f'|$ and $G_1(x,y)=(4\pi)^{-N/2}\sum_{k\in\Z^N}e^{-|x-Rk-y|^2/4}$. Since the function $G_1$ is bounded in $\R^N\times\R^N$ (being continuous and periodic), one gets that $0\le u(1,\cdot)\le e^L\|G_1\|_{L^\infty(\R^N\times\R^N)}\omega_Nr^N$ in~$\R^N$ (where $\omega_N$ is the Lebesgue measure of the unit ball $B_1$). As a consequence, for $r$ sufficiently small one has  $0\le u(1,\cdot)\le\alpha/2$ in $\R^N$ (with $\alpha>0$ as in~\eqref{bistable}) which readily implies that $\lim_{t\to+\infty}\|u(t,\cdot)\|_{L^\infty(\R^N)}=0$.

Subgraphs are another important class of unbounded sets $U$. While further flattening properties of level sets of solutions of~\eqref{homo}-\eqref{defu0}, when $U$ is a subgraph, and more precise estimates of their location are established in~\cite{HR3}, we here list direct applications of Theo\-rems~\ref{th1}-\ref{th4} for some specific subgraphs. Let us use $x=(x',x_N)\in\R^{N-1}\times\R$ for the generic notation of a point $x\in\R^N$, and let 
\be\label{Ugamma}
U=\big\{x\in\R^N:x_N\le\gamma(x')\big\},
\ee
with $\gamma:\R^{N-1}\to\R$ in $L^\infty_{loc}(\R^{N-1})$. Assume first that $\gamma$ is of the form
\be\label{hypgamma1}
\gamma(x')=\alpha\,|x'|+o(|x'|) \quad\text{as }\;|x'|\to+\infty,
\ee
for some $\alpha\in\R$, which is for instance the case if $\gamma\in C^1(\R^{N-1})$ and $\nabla\gamma(x')\.x'/|x'|\to\alpha$ as $|x'|\to+\infty$. We see that $U_\rho\neq\emptyset$ for any $\rho>0$ and that
$$\mc{B}(U)=\big\{e\in\Sph:e_N>\alpha|e'|\big\},\quad\mc{U}(U)=\mc{U}(U_\rho)=\big\{e\in\Sph:e_N\leq\alpha|e'|\big\}.$$
Thus~\eqref{hyp:U} is fulfilled and Theorems~\ref{th1}-\ref{th3} entail the validity of~\eqref{ass0},~\eqref{ass},~\eqref{ass-cpt} and~\eqref{Hloc} under Hypothesis~\ref{hyp} on~$f$. However, the shape of the spreading set $\mc{W}$ given by~\eqref{asspre-formula} is completely different according to the sign of $\alpha$. Namely, if $\alpha>0$ then $\mc{W}=\{x\in\R^N:x_N< \alpha\,|x'|+c^*\sqrt{1+\alpha^2}\}$ --~which is simply a translation of the interior of the cone $\R^+ \mc{U}(U)$~-- hence  $\mc{W}$ is non-convex and not $C^1$. If~$\alpha<0$ then the spreading set~$\mc{W}$ is still given by the $c^*$-neighborhood of the same cone $\R^+\mc{U}(U)$, which now becomes ``rounded'' in its upper part, namely~$w(e)=c^*$ if~$e_N\ge|e'|/|\alpha|$, and moreover $\mc{W}$ is convex and of class~$C^1$ (but not $C^2$). Finally, if~$\alpha=0$ (which includes the case $\gamma$ bounded) then $\mc{W}$ is given by the half-space $\{x\in\R^N:x_N<c^*\}$, i.e.~$w(e)=+\infty$ if~$e_N\le0$, and~$w(e)=c^*/e_N$ if~$e_N>0$. In the case where $\gamma$ satisfies
\be\label{hypgamma2}
\frac{\gamma(x')}{|x'|}\to-\infty\ \hbox{ as $|x'|\to+\infty$},
\ee
then $U_\rho\neq\emptyset$ for any $\rho>0$, $\mc{B}(U)\!=\!\Sph\setminus\{(0,\cdots,0,-1)\}$, $\mc{U}(U)\!=\!\mc{U}(U_\rho)\!=\!\{(0,\cdots,0,-1)\}$. Here again~\eqref{hyp:U} is fulfilled and therefore, under Hypothesis~\ref{hyp}, the conclusions of Theorem~\ref{th1}-\ref{th3} hold with $\Sph\ni e\mapsto w(e)$ having the envelop
$$\mc{W}=\R^+(0,\cdots,0,-1)+B_{c^*}=\big\{x\in\R^N:|x'|<c^*,\ x_N\leq0\big\}\cup B_{c^*}.$$
This is a cylinder with a ``rounded'' top, which is convex and $C^1$, but not $C^2$. Lastly, if $\gamma$ in~\eqref{Ugamma} is assumed to have uniformly bounded local oscillations, that is, if
\be\label{hypgamma3}
\sup_{x',y'\in\R^{N-1},\,|x'-y'|\le1}|\gamma(x')-\gamma(y')|<+\infty,\footnote{Condition~\eqref{hypgamma3} is fulfilled for instance if $\gamma$ is globally Lipschitz-continous; it may or may not be fulfilled if $\gamma$ satisfies~\eqref{hypgamma1}; it is not fulfilled if $\gamma$ satisfies~\eqref{hypgamma3}.}
\ee
then condition~\eqref{dUrho} is fulfilled, and Theorem~\ref{th4} then yields~\eqref{dH}.

To complete this section, we list some situations where one or both hypotheses~\eqref{hyp:U} and~\eqref{dUrho} of Theorems~\ref{th1}-\ref{th4} do not hold and the conclusions~\eqref{ass0},~\eqref{ass},~\eqref{ass-cpt}, \eqref{Hloc} and~\eqref{dH} fail (the examples will then also show that the conditions~\eqref{hyp:U} and~\eqref{dUrho} on~$U$ cannot be compared in general). We also further discuss the validity of the following convergences:
\Fi{Hausdorff}
\lim_{t\to+\infty}\frac1t\, E_\lambda(t)=\mc{W}=\lim_{t\to+\infty}\frac1t\,U+B_{c^*},
\Ff
that one may expect to hold but that actually fail in general.  The above convergences~\eqref{Hausdorff} would be understood with respect to the Hausdorff distance (which, we point out, does not guarantee the uniqueness of the limit). The precise counter-examples, which are given in Section~\ref{sec6}, will enlighten the sharpness of our results. We first observe that, whenever~\eqref{hyp:U} is fulfilled, together with $U_\rho\neq\emptyset$ and Hypothesis~\ref{hyp}, then~\eqref{Hloc} holds and therefore the limit of $t^{-1}E_\lambda(t)$, if any, must coincide with the set $\mc{W}$ (in the sense that the Hausdorff distance between the limit set and $\mc{W}$ must be~$0$). All of the following instances refer to the equation~\eqref{homo} with logistic term $f(u)=u(1-u)$, for which Hypothesis~\ref{hyp} holds, as well as the hair trigger effect, i.e., $\theta\in(0,1)$ and $\rho>0$ can be arbitrary in Proposition~\ref{pro1}. Then~\eqref{hyp:U} and~\eqref{dUrho} are understood here with~$\rho>0$ arbitrarily~small.

\begin{itemize}
\item There exists $U$, with non-empty interior, violating~\eqref{hyp:U}, but fulfilling~\eqref{dUrho} (hence~\eqref{dH} holds), for which \eqref{ass0}, \eqref{ass},~\eqref{ass-cpt} and~\eqref{Hloc} all fail, for any function $w:\Sph\to[0,+\infty]$ and any star-shaped, open set~$\mc{W}\subset\R^N$, and moreover both limits in~\eqref{Hausdorff} do not exist (see Proposition~\ref{pro:ce1}).		
\item There exists $U$, with non-empty interior, violating~\eqref{dUrho}, but fulfilling~\eqref{hyp:U} (hence~\eqref{ass0},~\eqref{ass},~\eqref{ass-cpt} and~\eqref{Hloc} hold), for which~\eqref{dH} fails and the first limit in~\eqref{Hausdorff} exists whereas the second one does not (see Proposition~\ref{pro:ce1bis}).	
\item There exists $U$, with non-empty interior, violating both~\eqref{hyp:U} and~\eqref{dUrho}, for which~\eqref{ass0},~\eqref{ass},~\eqref{ass-cpt},~\eqref{Hloc} and~\eqref{dH} all fail, with $w(e)$ and $\W$ given by~\eqref{FGgeneral}-\eqref{conv} and~\eqref{asspre}, and the two limits in~\eqref{Hausdorff} exist but do not coincide (see Proposition~\ref{pro:ce2}).
\item There exists $U$, with non-empty interior, fulfilling~\eqref{hyp:U} and~\eqref{dUrho} (hence \eqref{ass0}, \eqref{ass}, \eqref{ass-cpt}, \eqref{Hloc} and~\eqref{dH} all hold), for which~both limits in~\eqref{Hausdorff} do not exist and $d_{\mc{H}}(t^{-1}E_\lambda(t),\mc{W})=+\infty$ for all $\lambda\in(0,1)$ and $t>0$ (see Proposition~\ref{pro:ce3}).	
\end{itemize}

\subsubsection*{Outline of the paper}

Proposition~\ref{pro1}, together with some other auxiliary results on planar fronts and on the notion of invasion property, are shown in Section~\ref{sec3}. In Section~\ref{sec4} we show the convergences towards $1$ and towards $0$ on some sets of points. The former is directly deduced from Proposition~\ref{pro1}, while, for the latter, we make use of a family of supersolutions, constructed on the basis of the results of Section~\ref{sec3}, whose level sets are retracting spheres. The results of Section~\ref{sec4} are applied in Section~\ref{sec5} to prove Theorems~\ref{th1}-\ref{th4} on the spreading speeds and sets for general initial supports $U$. We finally exhibit in Section~\ref{sec6} some counter-examples when the geometric assumptions~\eqref{hyp:U} and~\eqref{dUrho} of Theorems~\ref{th1}-\ref{th4} are not satisfied.
	
%-----------------------------------------------------------------------------------
%-----------------------------------------------------------------------------------
			
\section{Preliminary considerations on planar fronts, and proof of Proposition~\ref{pro1}}\label{sec3}

This section is devoted to the proof of Proposition~\ref{pro1}, together with other auxi\-liary results on planar traveling fronts. We start with the definition of the ``invasion property", which means that the steady state~$1$ attracts the solutions of~\eqref{homo} --~not necessarily satisfying~\eqref{defu0}~-- that are initially ``large enough" in large balls: 

\begin{definition}\label{hyp:invasion}
We say that the {\rm invasion property} holds for~\eqref{homo} if there exist $\theta\in(0,1)$ and $\rho>0$ such that, if $u_0$ fulfills~\eqref{hyptheta} for some $x_0\in\R^N$, then the solution $u$ of~\eqref{homo} with initial condition $u_0$ satisfies $u(t,\cdot)\to1$ as $t\to+\infty$, locally uniformly in $\R^N$.
\end{definition}

From~\cite{AW,FM} (see also the comments after Proposition~\ref{pro1}), the invasion property is known to hold if $f$ is of the ignition type~\eqref{ignition}, or if $f>0$ in $(0,1)$ (as a consequence of the previous case and the maximum principle, by putting below $f$ a function of the ignition type), or if $f$ is of the bistable type~\eqref{bistable} with $\int_0^1f(s)ds>0$, or if $f$ is of the tristable case~\eqref{tristable} with $\int_0^\beta f(s)ds>0$ and $\int_\beta^1f(s)ds>0$. From~\cite{AW} and the maximum principle, the invasion property holds for every $\theta\in(0,1)$ and $\rho>0$ (namely, the hair trigger effect) if and only $f$ is positive in $(0,1)$ and satisfies~\eqref{HTconditions}. In that case, Proposition~\ref{pro1} implies that properties \eqref{c<c*}-\eqref{c>c*} hold for any compactly supported initial datum $0\leq u_0\leq 1$ such that $u_0>0$ on a set of positive measure (one knows that, in that case, $u(1,\cdot)$ is continuous and positive in $\R^N$, hence it satisfies~\eqref{hyptheta} with $x_0=0$, for some $\theta\in(0,1)$ and~$\rho>0$).

More generally speaking, it is known that the invasion property is equivalent to some simple conditions on the function~$f$ involving the integrals $\int_t^1f$ and the positivity of $f$ in a left neighborhood of~$1$, as stated in the following proposition.

\begin{proposition}[\cite{DP,P2}]\label{pro:hyp1}
For a $C^1([0,1])$ function $f$ such that $f(0)=f(1)=0$, the invasion property is equivalent to the following two conditions simultaneously:
\be\label{hyp:theta}
\exists\,\theta\in(0,1),\quad f>0\hbox{ in }[\theta,1),
\ee
and
\be\label{intf}
\forall\,t\in[0,1),\quad\int_t^1f(s)\,ds>0.
\ee
Furthermore, $\theta$ can be chosen as the same real number in Definition~$\ref{hyp:invasion}$ and in~\eqref{hyp:theta}.
\end{proposition}

The fact that the invasion property implies \eqref{hyp:theta}-\eqref{intf} is a consequence of~\cite[Proposition~2.12]{P2} and the converse implication follows from \cite[Lemma~2.4]{DP}. In particular, the invasion property is satisfied if $f\ge0$ in $[0,1]$ and if condition~\eqref{hyp:theta} holds. We also point out that Proposition~\ref{pro:hyp1} implies that the invasion property only depends on~$f$ and not on the dimension~$N$, whereas, for a function $f$ which is positive in $(0,1)$, the hair trigger effect (that is, the arbitrariness of $\theta\in(0,1)$ and $\rho>0$ in the invasion property) depends on~$N$ (for instance, for the function $f$ defined by $f(s)=s^p(1-s)$, with $p\ge1$, the invasion property holds in any dimension $N\ge1$, but the hair trigger effect holds if and only if $p\le1+2/N$, see~\cite{AW}).

The next three lemmas are part of the proof of Proposition~\ref{pro1}. The first one contains some standard properties of traveling fronts. We give its proof for the sake of completeness.

We recall that a traveling front connecting $1$ to~$0$ is a solution to~\eqref{homo} of the form $\varphi(x\cdot e-ct)$, for some $e\in\Sph$, $c\in\R$ and $\varphi:\R\to(0,1)$ satisfying $\varphi(-\infty)=1$ and $\varphi(+\infty)=0$. Namely, the profile $\vp\in C^2(\R)$ satisfies
\Fi{TF}
\begin{cases}
\varphi''(z)+c\varphi'(z)+f(\varphi(z))=0, &z\in\R\\
\varphi(-\infty)=1>\vp(z)>\varphi(+\infty)=0, &z\in\R.
\end{cases}
\Ff 

\begin{lemma}\label{lem:phi'<0}
Let $\vp\in C^2(\R)$ and $c\in\R$ satisfy~\eqref{TF}. Then $\vp'<0$ in~$\R$. Moreover $c>0$ if and only if~$\int_0^1f(s)\,ds>0$.
\end{lemma}

\begin{proof}
Let $\vp\in C^2(\R)$ and $c\in\R$ satisfy~\eqref{TF}. Let us first show that the positivity of $c$ is equivalent to that of~$\int_0^1f(s)\,ds$. Standard elliptic estimates imply that $\varphi'(\pm\infty)=\varphi''(\pm\infty)=0$. Hence, integrating the equation in~\eqref{TF} against $\varphi'$ over~$\R$~yields
$$\int_{\R}c\,(\varphi'(z))^2\,dz=\int_{0}^1f(s)\,ds,$$
which shows the desired property.

Let us show now the decreasing monotonicity of $\vp$. Suppose first that $c\ge0$. Assume by contradiction that $\varphi$ is not non-increasing. Then there are $x_m<x_M<y\in\R$ such that $\varphi(x_m)=\varphi(y)<\varphi(x_M)$, and~$x_m$ and $x_M$ are respectively a point of local minimum and of local maximum for $\varphi$. Integrating the equation in~\eqref{TF} against $\varphi'$ over $[x_m,y]$ leads to:
$$c\int_{x_m}^y(\varphi')^2=-\frac{(\varphi'(y))^2}2,$$
we deduce that $\varphi'(y)=0$ and $c=0$, and thus $\varphi$ is periodic by the Cauchy-Lipschitz theorem (since $\varphi(x_m)=\varphi(y)$ and $\varphi'(x_m)=\varphi'(y)=0$). This is impossible. Therefore, $\varphi$ is non-increasing. Differentiating the equation in~\eqref{TF} and applying the strong maximum principle to $\varphi'$, one eventually gets that $\varphi'<0$ in $\R$.

In the case $c<0$, one finds $y<x_m<x_M$ such that $\varphi(y)=\varphi(x_M)>\varphi(x_m)$, and~$x_m$ and $x_M$ are respectively a point of local minimum and of local maximum for $\varphi$. Integrating the equation in~\eqref{TF} against $\varphi'$ over $[y,x_M]$ leads to the same contradiction as~before.
\end{proof}

The next lemma shows that Hypothesis~\ref{hyp} implies the invasion property.

\begin{lemma}\label{lemhyp3}
Hypothesis~$\ref{hyp}$ implies~\eqref{hyp:theta}-\eqref{intf} $($and then the invasion property$)$.
\end{lemma}

\begin{proof}
By Hypothesis~\ref{hyp}, problem~\eqref{TF} admits a solution $\vp=\vp_0$ for $c=c_0>0$. We know from Lemma~\ref{lem:phi'<0} that property~\eqref{intf} holds for $t=0$. Next, since $\varphi_0'(-\infty)=\varphi_0''(-\infty)=0$ owing to elliptic estimates, integrating the first equation in~\eqref{TF} against $\varphi_0'$ over the interval $(-\infty,x)$, for any $x\in\R$, yields
$$\frac{(\varphi_0'(x))^2}{2}+\int_{-\infty}^xc_0\,(\varphi_0'(z))^2\,dz=\int_{\vp_0(x)}^{1}f(s)\,ds.$$
By the arbitrariness of $x$ and the positivity of $c_0$, we deduce that $\int_t^1f(s)\,ds>0$ for all~$t\in(0,1)$. Property~\eqref{intf} is thereby shown.

Let us turn to~\eqref{hyp:theta}. Assume by contradiction that this pro\-perty does not hold, that is, that there exists a sequence $\seq{t}$ in $(0,1)$ converging to $1$ such that $f(t_n)\leq0$ for all $n\in\N$. Together with~\eqref{intf}, it follows that there exists another sequence $\seq{\sigma}$ in $(0,1)$ converging to~$1$ such that $f(\sigma_n)=0$ for all $n\in\N$. We deduce in particular that $f'(1)=0$. For $n\in\N$, consider the function
$$\psi:z\mapsto\psi(z):=\sigma_n+e^{-c_0z/2},$$
where $c_0>0$ is, as in the previous paragraph, given by Hypothesis~\ref{hyp}. With $f$ being extended for convenience by $0$ in $(1,+\infty)$, we have that
\begin{align*}
\psi''(z)+c_0\psi'(z)+f(\psi(z)) &=-\frac{c_0^2}4 \,e^{-c_0z/2}+f(\sigma_n)+\int_{\sigma_n}^{\psi(z)}f'(s)\,ds\vspace{3pt}\\
&\leq-\frac{c_0^2}4 \,e^{-c_0z/2}+\bigg(\max_{s\in[\sigma_n,1]}f'(s)\bigg)e^{-c_0z/2}
\end{align*}
for all $z\in\R$. Taking $n$ large enough, we find that the right-hand side is negative, that is, $\psi$ is a strict supersolution of the equation satisfied by the front profile $\vp_0$. By suitable translation in $z$, one can reduce to the case where $\psi(z)$ ``touches from above'' $\vp_0(z)$, i.e., $\min_\R(\psi-\vp_0)=0$. This contradicts the elliptic strong maximum principle. Therefore,~\eqref{hyp:theta} is satisfied too.

As a conclusion, one has shown that Hypothesis~\ref{hyp} implies both~\eqref{intf} and~\eqref{hyp:theta}, hence it implies the invasion property from Proposition~\ref{pro:hyp1}.
\end{proof}

The last of the three lemmas gives the equivalence between the existence of a traveling front connecting $1$ to $0$ with a positive speed and the existence of a positive minimal speed for such fronts.

\begin{lemma}\label{lemminimal}
If problem~\eqref{TF} admits a solution for $c=c_0>0$ then there exists $c^*>0$ such that~\eqref{TF} admits a solution for $c=c^*$ and none for $c<c^*$.
\end{lemma}

\begin{proof}
Let us assume that~\eqref{TF} admits a solution $\vp=\vp_0$ for $c=c_0>0$, that is, that Hypothesis~\ref{hyp} holds. The first statement of Lemma~\ref{lem:phi'<0} implies that~$\int_0^1f(s)\,ds>0$, and thus that if~\eqref{TF} admits a solution for some $c\in\R$ then necessarily $c>0$. Let $c^*\geq0$ be the infimum of $c$ for which~\eqref{TF} admits a solution, and let $\seq{c}$ in $[c^*,+\infty)$ be such that~\eqref{TF} admits a solution $\vp=\vp_n$ for $c=c_n$, for any $n\in\N$, with $c_n\to c^*$ as $n\to+\infty$. The functions $\seq{\vp}$ are decreasing thanks to Lemma~\ref{lem:phi'<0}. We normalize them through horizontal translations by the condition $\vp_n(0)=\theta$, where $\theta\in(0,1)$ is the value provided by property~\eqref{hyp:theta}, which holds thanks to Lemma~\ref{lemhyp3}. By elliptic estimates, the sequence $\seq{\vp}$ converges (up to subsequence) locally uniformly towards a solution~$\vp^*$ of the first equation in~\eqref{TF} with $c=c^*$. The function $\vp^*$ is nonincreasing and satisfies $0\leq\vp^*\leq1$, $\vp^*(0)=\theta$ and moreover, by elliptic estimates, $(\vp^*)'(\pm\infty)=(\vp^*)''(\pm\infty)=0$. It follows that $f(\vp^*(\pm\infty))=0$, hence in particular $\vp^*(-\infty)=1$ due to~\eqref{hyp:theta}. Applying the strong maximum principle to the function $(\vp^*)'$, which solves a linear elliptic equation, we infer that $(\vp^*)'<0$ in $\R$.

We now show that $\vp^*(+\infty)=0$. To do this, we observe that, if $\vp'<0$ in~$\R$, then the equation in~\eqref{TF} can be reformulated in terms of the negative function $p(s):=\varphi'(\varphi^{-1}(s))$~as
$$p'(s)p(s)=-cp(s)-f(s)\for s\in(\vp(+\infty),\vp(-\infty)).$$
Moreover, by elliptic estimates, the function $p$ satisfies $p(\vp(+\infty)^+)=p(\vp(-\infty)^-)=0$. We consider such reformulation for the functions $p_0$ and $p^*$ associated with $\vp_0$ and $\varphi^*$ respectively, recalling that $(\vp^*)'<0$ in $\R$, and also $\vp_0'<0$ in $\R$ thanks to Lemma~\ref{lem:phi'<0}. We have that $\vp_0(-\infty)=1$ and $\vp_0(+\infty)=0$. Assume by way of contradiction that~$\ul q:=\vp^*(+\infty)>0$. Hence $p^*(\ul q^+)=p^*(1^-)=0$. We set
$$\ol q:=\sup\big\{q\in (\ul q,1)\ :\ p_0(s)<p^*(s)\text{ for all }s\in(\ul q,q)\big\}\in(\ul q,1].$$
It may happen that $\ol q=1$, but in any case $p_0(\ol q^-)=p^*(\ol q^-)$, and $p_0<p^*$ in $(\ul q,\ol q)$. We then integrate the equations satisfied by $p^*$ and $p_0$ over $(\ul q,\ol q)$, take their difference, and obtain
$$\frac12\,p_0^2(\ul q)=\int_{\ul q}^{\ol q}\big(c_0p_0(s)-c^*p^*(s)\big)\,ds.$$
Since $c_0>0$ and $c_0\geq c^*\ge0$ by the definition of $c^*$, while $p_0<p^*<0$ in $(\ul q,\ol q)$, we have reached a contradiction.

We have shown that $\vp^*(+\infty)=0$. Therefore, $\vp^*$ solves~\eqref{TF} with $c=c^*$ and thus $c^*>0$ by Lemma~\ref{lem:phi'<0}. By the definition of~$c^*$ there cannot be any solution to~\eqref{TF} with $c<c^*$. This concludes the proof.
\end{proof}

The last step before the proof of Proposition~\ref{pro1} is a key-result on the existence of front profiles in finite or semi-infinite intervals, for some speeds $c$ smaller or larger than the minimal speed $c^*$, under Hypothesis~\ref{hyp}. Before stating the result (which is also used in the proof of Theorems~\ref{th1}-\ref{th4}), we first recall that $f$ is extended by~$0$ outside~$[0,1]$, and that a planar front connecting $1$ to $0$ with speed $c\in\R$ is a solution $\varphi(x-ct)$ of~\eqref{homo} in~$\R$. The profile $\vp$ satisfies the equation 
\Fi{tw}
\varphi''+c\varphi'+ f(\varphi)=0,
\Ff
which is equivalent to the system of ODEs
\Fi{ODEs}
\begin{cases}
q'=p\\
p'=-cp-f(q).
\end{cases}
\Ff
In the phase plane, this system generates orbits $(q(x),p(x))$ which, as long as $p\neq0$, can be parameterized as the graph of a function $p=p(q)$ which solves
\Fi{phase}
\frac{dp}{dq}=-c-\frac{f(q)}p.
\Ff
Hence by Lemma~\ref{lem:phi'<0}, a solution to~\eqref{TF} corresponds to a heteroclinic connection between the stationary points $(0,0)$ and $(1,0)$ for~\eqref{ODEs}, along which~$p<0$. We always restrict to orbits of~\eqref{ODEs} contained in $[0,1]\!\times\!(-\infty,0]$. Hypothesis~\ref{hyp} and Lemma~\ref{lemminimal} translate to the existence of a heteroclinic connection between $(0,0)$ and $(1,0)$ for~\eqref{ODEs} when $c=c^*>0$ and the nonexistence of such connection when $c<c^*$. We also remember that, by Lemma~\ref{lem:phi'<0}, Hypothesis~\ref{hyp} implies that $\int_0^1f(s)\,ds>0$.

The following result asserts that, when~$c$ is slightly below the threshold $c^*$ for the existence of a heteroclinic connection for~\eqref{ODEs}, one can find a trajectory joining $(0,1)\times\{0\}$ to $\{0\}\times(-\infty,0)$, whereas, for $c$ above that threshold, there exists a trajectory joining $\{1\}\times(-\infty,0)$ to $\{(0,0)\}$.

\begin{proposition}\label{pro:fronts}
Assume that Hypothesis~$\ref{hyp}$ holds, and let $c^*>0$ be given by Lemma~$\ref{lemminimal}$. Then the following properties hold:
\begin{itemize}
\item[{\rm{(i)}}] there is $\eta\in(0,c^*)$ such that, for any $c\in[c^*-\eta,c^*)$, there exists a $C^2$ decreasing function $\ul\varphi$ defined in some interval $[0,a]$, with $a\!>\!0$, satisfying~\eqref{tw} in~$[0,a]$ together~with 
$$\theta<\ul\varphi(0)<1, \qquad \ul\varphi'(0)=0,\qquad \ul\varphi(a)=0,$$
where $\theta\in(0,1)$ is the value in condition~\eqref{hyp:theta}, which holds thanks to Lemma~$\ref{lemhyp3}$;
\item[{\rm{(ii)}}] for any $c>c^*$, there exists a decreasing function $\ol\varphi\in C^2(\R)$ satisfying~\eqref{tw} in~$\R$ together with 
\Fi{olphi}
\ol\varphi(0)=1,\ \ \ol\varphi(+\infty)=0,\ \ 0<m^{-1}\ol\varphi\leq-\ol\varphi'\leq m\,\ol\varphi\hbox{ in }\R,\ \ \ol\varphi''\geq0\hbox{ in }[b,+\infty),
\Ff 
for some $m>1$ and $b>0$.	
\end{itemize}
\end{proposition}
	
\begin{proof}
To start with, one observes that $c^*\geq 2\sqrt{f'(0)}$ if $f'(0)>0$. Indeed, otherwise $(0,0)$ is a focus for~\eqref{ODEs} when $c=c^*$ and then no trajectory can converge toward it, whereas one such trajectory is provided by the solution to~\eqref{TF} for $c=c^*$, which exists by Lemma~\ref{lemminimal}. In order to prove the statements~(i) and~(ii), we first show that $c^*$ coincides with the critical speed $\hat c$ constructed in~\cite{AW}. 
		
Let us recall the construction of $\hat c$ in~\cite{AW}. First, for any $c\in\R$ and any $\e>0$, let $p=p_{c,\e}(q)$ be the trajectory of~\eqref{ODEs} emerging from the regular point $(0,-\e)$. Namely, $p_{c,\e}(q)$ is a negative solution to~\eqref{phase} for non\-negative~$q$ in a maximal interval $[0,q_{c,\e})$, with $q_{c,\e}\in(0,+\infty]$, together with $p_{c,\e}(0)=-\e$. Hence~$p_{c,\e}^2$ is a positive solution to
\Fi{pc2}
(p_{c,\e}^2)'(q)=-2cp_{c,\e}(q)-2f(q)
\Ff
in $[0,q_{c,\e})$. If $q_{c,\e}>1$ then  $p_{c,\e}<0$ in $[0,1]$. In the case $q_{c,\e}\leq 1$, one has that $p_{c,\e}$ is bounded in $[0,q_{c,\e})$ (because $(p_{c,\e}^2)'\le|c|p_{c,\e}^2+|c|+2\max_{[0,1]}|f|$ in $[0,q_{c,\e})$) and then, by~\eqref{pc2}, $p_{c,\e}^2$ is Lipschitz-continuous in $[0,q_{c,\e})$ and satisfies $\lim_{q\nearrow q_{c,\e}}p_{c,\e}^2(q)=0$; in such a case we then set $p_{c,\e}=0$ in $[q_{c,\e},1]$. The functions $p_{c,\e}$ are nonincreasing with respect to $\e$ in $[0,1]$. Moreover the family $(p_{c,\e}^2)_{\e\in(0,1]}$ is bounded in $C^{0,1}([0,1])$. One then lets $p_c$ denote the limit in $[0,1]$ as $\e\searrow0$ of $p_{c,\e}$. It follows that $p_c(0)=0$, that $p_c\leq 0$ in $[0,1]$ and that~$p_c^2$ is Lipschitz-continuous in $[0,1]$. Moreover, as long as $p_c$ is negative, it parameterizes a trajectory of~\eqref{ODEs}, that is, it solves~\eqref{phase} hence~$p_c^2$ solves~\eqref{pc2}. For each $\e>0$ and $c<c'$, equation~\eqref{phase} shows that $p_{c,\e}>p_{c',\e}$ in a right neighborhood of $0$ and also that there is no $q\in(0,1]$ such that $p_{c,\e}>p_{c',\e}$ in $(0,q)$ and  $p_{c,\e}(q)=p_{c',\e}(q)<0$. It follows that $p_{c,\e}\geq p_{c',\e}$ in $[0,1]$ and therefore 
\Fi{pc>pc'}
p_{c}\geq p_{c'} \ \ \text{in $[0,1],\ $ for any $c<c'$}.
\Ff 
On the other hand, if $c>1+\max_{[0,1]}|f'|$ one has, for each $\e>0$, $p_{c,\e}(q)\leq -q$ for all $q\in[0,1]$, because otherwise there would be $q_0\in(0,1]$ such that $p_{c,\e}(q)<-q$ in $[0,q_0)$ and $p_{c,\e}(q_0)= -q_0$. This would contradict~\eqref{phase}. Therefore $p_{c}(q)\leq -q$ for all $q\in[0,1]$. Then the critical speed~$\hat c$ is defined as the infimum of $c$ such that~$p_c(1)<0$.
		
Let us show that $\hat c\leq c^*$. Let $p^*(q)$ be the parameterization of a front with (minimal) speed $c^*$ given by Lemma~\ref{lemminimal}. Then $p^*(q)$ is a negative solution of~\eqref{phase} in $(0,1)$ and (its continuous extension to $[0,1]$) satisfies $p^*(0)=p^*(1)=0$. By construction, $p_{c^*}=\lim_{\e\searrow0}p_{c,\e}\leq p^*$ in $[0,1]$ (trajectories with the same $c$ cannot cross each other), hence $p_{c^*}<0$ in~$(0,1)$. We shall now prove that
\be\label{pc<0}
\forall\,c>c^*,\quad p_c<0\ \hbox{ in }(0,1].
\ee
This would imply in particular that $p_c(1)<0$ for all $c>c^*$, hence, by definition, $\hat c\le c^*$. To show~\eqref{pc<0}, we first infer from~\eqref{pc>pc'} that, for $c>c^*$, one has $p_{c}\leq p_{c^*}\le p^*<0$ in $(0,1)$ and thus both $p_c$ satisfies~\eqref{TF} in $(0,1)$. Next, if by contradiction $p_c(1)=0$, then integrating the equations satisfied by $p_c^2$ and $(p^*)^2$ over $(0,1)$ and taking the difference~we~get
$$\int_{0}^{1}\big(cp_c(s)-c^*p^*(s)\big)\,ds=0.$$ 
This is impossible because $c>c^*>0$ and $p_{c}\leq p^*<0$ in $(0,1)$.

Let us now show that $\hat{c}\ge c^*$. Assume by contradiction that $\hat{c}<c^*$. Then there exists $c\in(\hat c,c^*)$ for which $p_c(1)<0=p^*(1)$. If by contradiction $p_c(q_1)\geq p^*(q_1)$ for some $q_1\in(0,1)$ then there would exist $q_2\in[q_1,1)$ such that 
$$p_c(q_2)=p^*(q_2),\qquad\frac{dp_c}{dq}(q_2)\leq \frac{dp^*}{dq}(q_2),$$
which is impossible due to \eqref{phase}. This means that $p_c<p^*<0$ in~$(0,1)$. Now, for $\e>0$, let $\t p_{c,\e}(q)$ be the parametrization of the trajectory emerging from the regular point $(1,-\e)$. Since $\t p_{c,\e}(1)<0=p^*(1)$, the same argument as before shows that $\t p_{c,\e}(q)<p^*(q)$ for all $q\in(0,1]$. In particular, $\t p_{c,\e}$ satisfies the equation~\eqref{phase} in $(0,1]$. On the other hand, for $\e<-p_c(1)$ we have that $\t p_{c,\e}(1)=-\e>p_c(1)$ and thus by uniqueness for~\eqref{phase} one gets $\t p_{c,\e}>p_c$ in the whole $(0,1]$, and for the same reason the functions $\t p_{c,\e}$ are decreasing with respect to $\e$ at any point in $(0,1]$. Furthermore, the arguments used for the functions $p_{c,\e}$ imply that the family $(\t p_{c,\e}^2)_{\e\in(0,1]}$ is bounded in $C^{0,1}([0,1])$. As a consequence, $\t p_{c,\e}$ converges  as $\e\searrow0$ to a function~$\tilde p_c$  uniformly in $[0,1]$. The function $\t p_c$ is continuous in~$[0,1]$, satisfies $p_c\leq\tilde p_c\leq p^*$ in $[0,1]$ and vanishes at $1$ (and also at $0$ since $p_c$ and $p^*$~do). Moreover it solves~\eqref{phase} in $(0,1)$. Therefore~$\t p_c$ parameterizes a trajectory of~\eqref{ODEs} connecting~$(1,0)$ to~$(0,0)$, which then corresponds to a solution of~\eqref{TF} with speed~$c<c^*$ contradicting Lemma~\ref{lemminimal}. We have thereby shown that~$\hat c=c^*$.
		
We can now prove statement (i). Let us fix $\bar\e>0$. By Cauchy-Lipschitz theorem one has $p_{c^*,\bar\e}<p^*\leq0$ in $[0,1)$. Integrating the equation satisfied by $(p^*)^2$ and $p_{c^*,\bar\e}^2$ over $(0,1)$ and taking the difference we get
$$\frac12p_{c^*,\bar\e}^2(1)=\frac{1}2\bar\e^2+c^*\int_{0}^{1}\big(p^*(s)-p_{c^*,\bar\e}(s)\big)\,ds>0,$$ 
hence $p_{c^*,\bar\e}(1)<0$. Thus it holds that $p_{c^*,\bar\e}<p^*\leq0$ in $[0,1]$. It follows that $p_{c^*,\bar\e}$ satisfies~\eqref{phase} in $[0,1]$ and thus there exists $\eta\in(0,c^*)$ such that $p_{c,\bar\e}<p^*$ in $[0,1]$ for all $c\in[c^*-\eta,c^*)$. Take one of such $c$'s. We know that $p_c(0)=0$ and also that $p_c(1)=0$ (because $c<c^*$ and the latter coincides with $\hat c$) therefore there exists $q_c\in(0,1)$ such that $p_c(q_c)=0>p^*(q_c)$ (because otherwise $p_c$ would give rise to a solution of~\eqref{TF} with speed $c<c^*$, which is impossible due to Lemma~\ref{lemminimal}). We then define
$$\e_c:=\inf\{\e\in(0,\bar\e]\ :\ p_{c,\e}<p^* \text{ in }[0,\max(q_c,\theta)]\},$$
where $\theta\in(0,1)$ is given by~\eqref{hyp:theta}, which holds thanks to Lemma~$\ref{lemhyp3}$. Since $\lim_{\e\searrow0}p_{c,\e}(q_c)=p_c(q_c)>p^*(q_c)$ one deduces that $\e_c>0$. Moreover, since $p_{c,\e}$ solves~\eqref{phase} in the region where it is negative, the continuous dependence with respect to $\e$ implies that $p_{c,\e_c}\leq p^*$ in $[0,\max(q_c,\theta)]$ with equality at some point $q_c'\in[0,\max(q_c,\theta)]$. Clearly $q'_c>0$, furthermore $q'_c\notin (0,\max(q_c,\theta))$ because otherwise we would have both $p_{c,\e_c}(q'_c)=p^*(q'_c)<0$ and $p_{c,\e_c}'(q'_c)=(p^*)'(q'_c)$, which is ruled out by the equations since~$c<c^*$. As a consequence $q'_c=\max(q_c,\theta)$. From this one gets $p_{c,\e_c}'(q'_c)>(p^*)'(q'_c)$ again by the equation. We finally claim that there exists $s_c\in(q'_c,1)$ such that $p_{c,\e_c}<0$ in $[0,s_c)$ and $p_{c,\e_c}(s_c)=0$. Otherwise there would be~$q''_c\in (q'_c,1]$ such that $p^*<p_{c,\e_c}<0$ in $(q'_c,q''_c)$ and $p_{c,\e_c}(q''_c)=p^*(q''_c)$ which, by the usual integration, leads to the contradiction
$$c\int_{q'_c}^{q''_c}p_{c,\e_c}(s)\, ds= c^*\int_{q'_c}^{q''_c}p^*(s)\, ds.$$
The parametrization $p_{c,\e_c}$ provides the desired function~$\ul\varphi$ (up to translation) with $\ul\vp(0)=s_c\in(\theta,1)$.

We finally prove (ii). Take $c>c^*$. We know that $p_c$ vanishes at $0$ and, by~\eqref{pc<0}, that $p_c$ is negative in $(0,1]$. Recalling that $f$ is extended to $0$ in $(1,+\infty)$, we extend $p_c$ in $(1,+\infty)$ by solving the equation~\eqref{phase} for $q>1$. Namely, $p_c(q)=p_c(1)-c(q-1)$ for $q>1$. We then define $\ol\varphi$ as the positive solution of~\eqref{tw} associated with this trajectory  satisfying~$\ol\varphi(0)=1$. One has that~$\ol\varphi(z)=Ae^{-cz}+1-A$ for $z<0$, for some $A>0$, and that $\ol\vp$ is defined on the whole~$\R$. The existence of $m>1$ such that~$|\ol\varphi'|\leq m\ol\varphi$ follows from elliptic estimates and Harnack's inequality. Next, by~\cite[Proposition~4.1]{AW}, $p_c(q)/q\leq-c/2$ in a right neighborhood  of~$0$, that is, $\ol\varphi'(z)\leq-(c/2)\ol\varphi(z)$ for~$z$ large enough. From this we deduce on the one hand that, up to increasing~$m$ if need be,~$\ol\varphi'\leq-m^{-1}\ol\varphi$ in $\R$. On the other hand, we infer that, for any $\nu>0$,
$$\ol\varphi''\geq \frac{c^2}2\ol\varphi-f(\ol\varphi)\geq \Big(\frac{c^2}2-f'(0)-\nu\Big)\ol\varphi,$$
for all $z$ large enough. Recalling that
$c>c^*\geq 2\sqrt{\max(f'(0),0)}$, we eventually get that $\bar\varphi''>0$ in some interval $[b,+\infty)$ with $b>0$.	
\end{proof}	

Putting together the previous results we derive Proposition~\ref{pro1}, which provides the asymptotic speed of spreading for solutions to~\eqref{homo} with compactly supported initial data.
		
\begin{proof}[Proof of Proposition~$\ref{pro1}$] 
Since we here assume Hypothesis~\ref{hyp}, Lemma~\ref{lemminimal} yields the exis\-tence of a positive minimal speed $c^*$ of traveling fronts for~\eqref{homo} connecting $1$ to $0$, and Lemma~\ref{lemhyp3} shows that~\eqref{hyp:theta}-\eqref{intf} and the invasion property are fulfilled, for some $\theta\in(0,1)$ and~$\rho>0$. Therefore, any solution $u$ emerging from an initial datum $u_0$ satisfying~\eqref{hyptheta} spreads, in the sense that~$u(t,\cdot)\to1$ as $t\to+\infty$ locally uniformly in $\R^N$, and moreover using the function $\ul\vp$ provided by Proposition~\ref{pro:fronts}~(i), exactly as in the proof of~\cite[Theorem~5.3]{AW}, one shows that such a spreading solution $u$ satisfies~\eqref{c<c*}. Assume now that the initial condition~$u_0$ of~\eqref{homo} is compactly supported. As~in the proof of~\cite[Theorem~5.1]{AW}, it follows from Proposition~\ref{pro:fronts}~(ii) that~\eqref{c>c*} holds. 
\end{proof}

\begin{remark}\label{rem:spreadingspeed}{\rm 
Under the sole invasion property, the property~\eqref{c<c*} of Proposition~\ref{pro1} is still fulfilled, for a certain positive speed $c^*$. Indeed, if $u_0$ is as in Definition~$\ref{hyp:invasion}$ and if $v$ denotes the solution to~\eqref{homo} with initial condition $v_0=\theta\,\1_{B_\rho(x_0)}$, then $v(t,\cdot)\to1$ as $t\to+\infty$ locally uniformly in $\R^N$ and there exists $T>0$ such that $1\ge u(T,\cdot+y)\ge v(T,\cdot+y)\ge v_0$ in $\R^N$ for every~$|y|\leq1$. Hence, iterating and using the comparison principle, one finds $1\ge u(kT+t,\cdot+ky)\ge v(kT+t,\cdot+ky)\ge v(t,\cdot)$ in $\R^N$ for all $k\in\N$, $t\geq0$, and $|y|\le1$. Since $v(t,\cdot)\to1$ locally uniformly as $t\to+\infty$, one readily infers that $\min_{|x|\le c t}u(t,x)\to1$ as $t\to+\infty$, for every $c\in[0,1/T)$.}
\end{remark}

%-----------------------------------------------------------------------------------
%-----------------------------------------------------------------------------------
	
\section{Sets of convergence towards $1$ and $0$}\label{sec4}

Loosely speaking, the quantity $w(e)$ defined by~\eqref{ass0} separates the region where $u$ converges to $1$ to the one where $u$ converges to $0$ as $t\to+\infty$. In this section, we obtain some sets of points which belong to one or the other regions. 

The set where $u\to1$ is immediately obtained from property~\eqref{c<c*} of Proposition~\ref{pro1}.

\begin{lemma}\label{lem:U+B}
Assume that Hypothesis $\ref{hyp}$ holds, let $c^*>0$ and $\rho>0$ be given by Proposition~$\ref{pro1}$, and let~$u$ be the solution of~\eqref{homo}-\eqref{defu0}, with $U\subset\R^N$ satisfying $U_\rho\neq\emptyset$. Then, it holds that
$$\forall\,c\in(0,c^*),\quad\inf_{x\in U_\rho+B_{ct}}\!u(t,x)\to1\as t\to+\infty.$$
\end{lemma}

\begin{proof}
Let $v$ be the solution to~\eqref{homo} emerging from the initial datum $v_0=\1_{B_\rho}$. Take $c\in(0,c^*)$ and $\lambda<1$. By~\eqref{c<c*} in Proposition~\ref{pro1}, there exists $T>0$ such that
$$\forall\,t\geq T,\ \ \forall\,x\in  B_{ct},\quad v(t,x)>\lambda.$$
Now, for any $x_0\in U_\rho$, it holds that $u_0\geq v_0(\.-x_0)$ in $\R^N$ and therefore, by the parabolic comparison principle,
$$\forall\,t\geq T,\ \ \forall\,x\in  B_{ct}(x_0),\quad u(t,x)\geq v(t,x-x_0)>\lambda.$$
This is true for any $x_0\in U_\rho$, hence the result follows from the arbitrariness of $\lambda<1$.
\end{proof}

Unlike the previous case, the set where $u\to0$ cannot be obtained from Proposition~\ref{pro1}. Instead, we make use of a new type of supersolutions whose level sets are retracting spheres. In the strong Fisher-KPP case (i.e.~when $s\mapsto f(s)/s$ is positive and nonincreasing in~$(0,1)$), such supersolutions could be obtained as the sums of a finite number of solutions, see \cite[Lemma~4.2]{HR2}. In order to handle the general case, we exploit the planar solutions provided by Proposition~\ref{pro:fronts}~(ii). Actually, the functions we will construct are supersolutions to~\eqref{homo} in a {\em generalized} sense: namely, they are functions $v\in C^0([0,T]\times\R^N)$, with $T>0$, such that, if a solution~$u$ to~\eqref{homo} satisfies $0\le u(0,\.)\le\psi(0,\cdot)$ in~$\R^N$, then $u(t,\cdot)\le\psi(t,\cdot)$ in $\R^N$ for all~$t\in[0,T]$. This is the key technical result of the paper.

\begin{proposition}\label{proradial}
Assume that Hypothesis $\ref{hyp}$ holds, and let $c^*>0$ be given by Proposition~$\ref{pro1}$. Then, for any $c>c^*$ and~$\lambda>0$, there exist $R>0$ $($depending on $f$, $N$, $c$ and $\lambda$$)$ and a family of functions~$(v^T)_{T>0}$ such that, for each~$T>0$,~$v^T$~is a positive generalized supersolution to~\eqref{homo} in $[0,T]\times\R^N$ and satisfies
\Fi{vT10}\left\{\baa{ll}
v^T(0,x)\geq 1, & \forall\,|x|\geq R+cT,\vspace{3pt}\\
v^T(t,0)<\lambda, & \forall\,t\in[0,T].\eaa\right.
\Ff
\end{proposition}
 
\begin{proof}
We start with constructing the desired family of supersolutions in dimension~$1$. We then use them to construct radially symmetric supersolutions in higher dimension. But before doing so, we introduce some auxiliary notations. For any
$$c'>c''>c^*,$$
consider the function $\ol\varphi$ provided by Proposition~\ref{pro:fronts}~(ii) associated with $c''$. Let $m>1$ be given by~\eqref{olphi}, and let $s_0\in(0,1)$ be such that
\Fi{f>f'2}
\forall\,s\in(0,s_0),\quad |f(s)-f'(0)s|\leq\frac{c'-c''}{4m}\, s.
\Ff
Call then $Z>0$ the quantity where $\ol\vp(Z)=s_0$. For $\beta>0$ we define 
$$\psi(z):=\ol\vp(z)\,e^{-\beta(z-Z)}\ \hbox{ for }z\in\R.$$
This function $\psi$ is of class $C^2(\R)$ and it satisfies in $\R$
$$-\psi''-(c'-\beta)\psi'= \big(f(\ol\vp)-(c'-3\beta-c'')\ol\vp'+\beta(c'-2\beta)\ol\vp\big)\,e^{-\beta(z-Z)}.$$
Then, because of~\eqref{olphi}, we can choose $\beta\in(0,c'-c'')$ small enough so that
\Fi{psisuper}
-\psi''-(c'-\beta)\psi'> f(\ol\vp)\,e^{-\beta(z-Z)}+\frac{c'-c''}{2m}\psi.
\Ff
With $b>0$ as in~\eqref{olphi}, we also choose arbitrarily large real numbers $L$ and $R'$ such that
\Fi{defR0}
L>\max\Big(Z+\frac{\log 2}{\beta},b\Big)\ \hbox{ and }\ R'>\frac{N-1}{\beta}.
\Ff
 
\medskip
\noindent{\it{Step 1: the $1$-dimensional case.}} Our goal is to connect $\ol\vp$ with its reflection $\ol\vp(-\cdot)$, by using an even function which is steeper than $\ol\vp$ at some point. This will be achieved through the function $\psi$ defined above. Then, the symmetrized function $\psi(x-c't)+\psi(-x-c't)$ will be a supersolution where it is smaller than $s_0$, and we take the minimum between suitable translations of the functions $\ol\vp(x-c't)$ and $\psi(x-c't)+\psi(-x-c't)$, which will be a (generalized) supersolution for $x\leq0$. Next, we want the minimum to be achieved by the latter function at $x=0$, so that we can extend the supersolution to the whole line by even reflection.

More precisely, we consider an arbitrary $T>0$ and we call
\Fi{defv12}\left\{\baa{ll}
v_1(t,r) & \!\!\!:=\ol\vp(r-c'(t-T)+L),\vspace{3pt}\\
v_2(t,r) & \!\!\!:=\psi(r-c'(t-T)+L)+\psi(-r-c'(t-T)+L),\vspace{3pt}\\
v(t,r) & \!\!\!:=\min\big(v_1(t,r),v_2(t,r)\big),\eaa\right.
\Ff
for $(t,r)\in[0,T]\times\R$, see figure \ref{fig:retracting_super}.
%\vspace{-7pt}
\begin{figure}[H]
	\begin{center}
		\includegraphics[height=4cm]{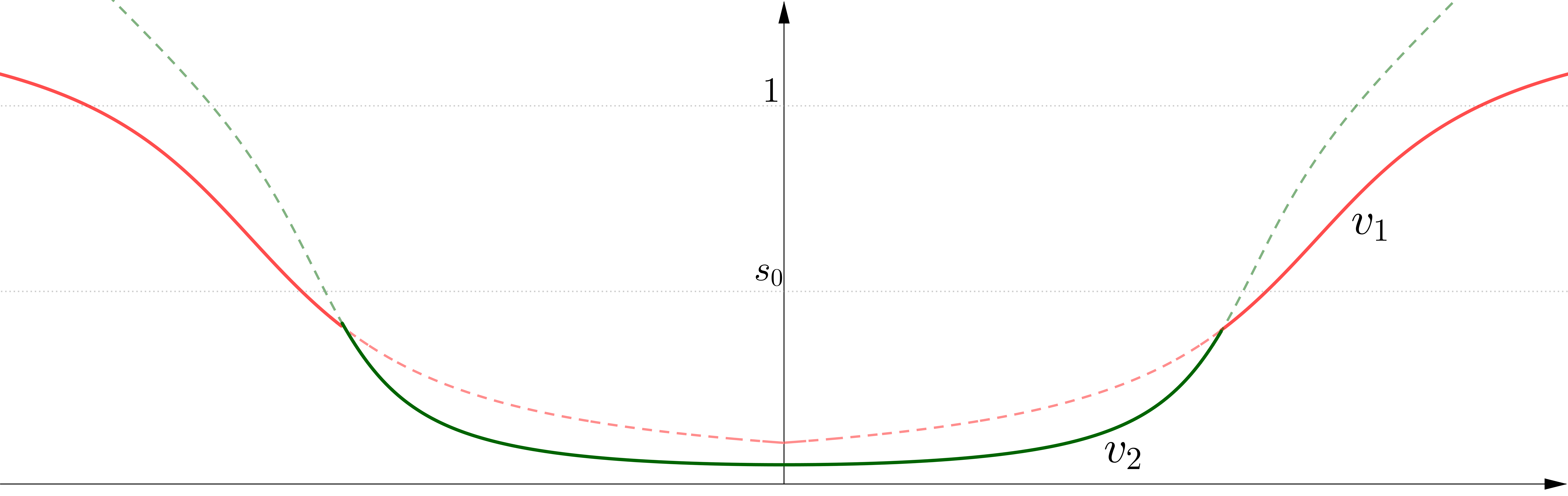}
		\caption{The definition of the generalized supersolution~$v=\min\big(v_1,v_2\big)$.}
		\label{fig:retracting_super}
	\end{center}
	\vspace{-7pt}
\end{figure}
%\vspace{-7pt}
The functions $v_1,v_2$ are positive. Moreover, we see that
$$\partial_tv_1(t,r)-\partial_{rr}v_1(t,r)-f(v_1(t,r))=(c''-c')\,\ol\vp'(r-c'(t-T)+L)>0\ \hbox{ in }[0,T]\times\R$$
since $c''<c'$ and $\ol\vp'<0$ in $\R$, hence $v_1$ is a supersolution to~\eqref{homo} in $[0,T]\times\R$.

The definition of $\psi$ and the positivity of $\ol\vp$ also imply that
$$\forall\,0\le t\le T,\ \forall\,r\leq c'(t-T)-L+Z,\quad v_2(t,r)>v_1(t,r).$$
This means that if there exists $(\bar t,\bar r)\in[0,T]\times(-\infty,0]$ where $v(\bar t,\bar r)=v_2(\bar t,\bar r)$, then necessarily $\bar r>c'(\bar t-T)-L+Z$. Together with the fact that $\ol\vp$ is decreasing, it follows that, for all $(\bar t,\bar r)\in[0,T]\times(-\infty,0]$,
\Fi{<s0}
v(\bar t,\bar r)=v_2(\bar t,\bar r) \implies  \left\{\baa{l}
0<v_2(\bar t,\bar r)\leq v_1(\bar t,\bar r)=\ol\vp(\bar r-c'(\bar t-T)+L)<s_0,\vspace{3pt}\\
0<\ol\vp(-\bar r-c'(\bar t-T)+L)\le\ol\vp(\bar r-c'(\bar t-T)+L)<s_0.\eaa\right.
\Ff
On the other hand, by~\eqref{psisuper} and the negativity of $\psi'$ we have that
$$\baa{rcl}
\partial_t v_2(\bar t,\bar r)-\partial_{rr}v_2(\bar t,\bar r)-\beta|\partial_{r}v_2(\bar t,\bar r)| & > & f(\ol\vp(\bar r-c'(\bar t-T)+L))\,e^{-\beta(\bar r-c'(\bar t-T)+L-Z)}\vspace{3pt}\\
& & +f(\ol\vp(-\bar r-c'(\bar t-T)+L))\,e^{-\beta(-\bar r-c'(\bar t-T)+L-Z)}\vspace{3pt}\\
& & \displaystyle+\frac{c'-c''}{2m}\,v_2(\bar t,\bar r).\eaa$$
Hence, estimating $f(\ol\vp(\pm\bar r-c'(\bar t-T)+L))$ by~\eqref{f>f'2}, and then using~\eqref{f>f'2} again, we eventually derive
\Fi{v2super}
\partial_t v_2-\partial_{rr}v_2-\beta|\partial_{r}v_2|>f(v_2)\quad 
\text{in }\;\big\{(t,r)\in[0,T]\times(-\infty,0]:v(t,r)=v_2(t,r)\big\}.
\Ff
At the point $r=0$ we compute, for $0\leq t\leq T$, 
$$v_2(t,0)=2\,\ol\vp(-c'(t-T)+L)\,e^{-\beta(-c'(t-T)+L-Z)}\leq 2\,v_1(t,0)\,e^{-\beta(L-Z)}.$$
Since $\beta(L-Z)>\log2$ by~\eqref{defR0}, we have that
\Fi{vt0}
\forall\,t\in[0,T],\ \ v(t,0)=v_2(t,0)<v_1(t,0)=\ol\vp(-c'(t-T)+L)\leq\ol\vp(L).
\Ff
Observe that the function $r\mapsto v_2(t,r)$ is even  and that $v(t,r)$ is equal to $v_2(t,r)$, hence symmetric with respect to~$r$, in a neighborhood of $r=0$, for each $0\leq t\leq T$. We deduce 
\Fi{v't0}
\forall\,t\in[0,T],\ \ \partial_r v(t,0)=\partial_r v_2(t,0)=0.
\Ff
Remember also that $v_1$ is a supersolution to~\eqref{homo} in~$[0,T]\times\R$. All these facts imply that, if we restrict~$v(t,r)$ to $r\leq0$ and we take its even reflection around $r=0$, we obtain a generalized supersolution to~\eqref{homo} in $[0,T]\times\R$, being the minimum of two classical supersolutions. We also see that, for every $0\leq t\leq T$,
\Fi{v''t0}\baa{rl}
\partial_{rr} v(t,0) & \!\!\!=\partial_{rr} v_2(t,0)=2\,\psi''(-c'(t-T)+L)\vspace{3pt}\\
& \!\!\!=2\,\big(\ol\vp''(-c'(t\!-\!T)\!+\!L)-2\beta\ol\vp'(-c'(t\!-\!T)\!+\!L)+\beta^2\ol\vp(-c'(t\!-\!T)\!+\!L)\big)\vspace{3pt}\\
& \qquad\qquad\times e^{-\beta(-c'(t-T)+L-Z)}\vspace{3pt}\\
& \!\!\!>0\eaa
\Ff
since $\ol\vp'<0$, $\ol\vp>0$ and $L>b$ by~\eqref{defR0}, where $b>0$ from~\eqref{olphi} is such that $\ol\vp''\geq 0$ in~$[b,+\infty)$.

\medskip
\noindent{\it{Step 2: the case of dimension $N\ge 2$.}} Consider the function $v$ defined before. With~$R'>0$ given by~\eqref{defR0}, we define, for $T>0$, a continuous function $v^T$ in $[0,T]\times\R^N$ as follows:
$$v^T(t,x):=\begin{cases}
v(t,0) & \text{if }t\in[0,T]\hbox{ and }|x|\leq R',\vspace{3pt}\\
v(t,R'-|x|) & \text{if }t\in[0,T]\hbox{ and }|x|> R'.\\
\end{cases}$$
We want to show that $v^T$ is a (generalized) supersolution to~\eqref{homo} in $[0,T]\times\R^N$.

We start with checking this in the region $|x|>R'$. Recall that $v$ is defined in~\eqref{defv12} as the minimum between~$v_1$ and~$v_2$. A direct computation reveals that the function~$v_1(t,R'-|x|)<1$ fulfills
$$\baa{l}
(\partial_t-\Delta) \big(v_1(t,R'-|x|)\big)-f(v_1(t,R'-|x|))\vspace{3pt}\\
\qquad\displaystyle=\partial_t v_1(t,R'-|x|)-\partial_{rr}v_1(t,R'-|x|)+\frac{N-1}{|x|}\partial_r v_1(t,R'-|x|)-f(v_1(t,R'-|x|))\vspace{3pt}\\
\qquad\displaystyle=\bigg(c''-c'+\frac{N-1}{|x|}\bigg)\,\ol\vp'(R'-|x|-c'(t-T)+L)\vspace{3pt}\\
\qquad>0\eaa$$
for all $t\in[0,T]$ and $|x|>R'$, since $\ol\vp'<0$ in $\R$ and $c''-c'+(N-1)/R'<0$ by~\eqref{defR0} together with $0<\beta<c'-c''$. We now turn to the function $v_2$. At any point $(\bar t,\bar x)\in[0,T]\times(\R^N\setminus\overline{B_{R'}})$ such that $v(\bar t,R'-|\bar x|)=v_2(\bar t,R'-|\bar x|)$, we deduce from~\eqref{defR0} and~\eqref{v2super} that
$$\baa{l}
(\partial_t-\Delta)\big(v_2(t,R'-|x|)\big)\big|_{\bar t,\,\bar x}-f(v_2(\bar t,R'-|\bar x|))\vspace{3pt}\\
\qquad\qquad\qquad\qquad\displaystyle>\beta\,|\partial_r v_2(\bar t,R'-|\bar x|)|+\frac{N-1}{|\bar x|}\,\partial_r v_2(\bar t,R'-|\bar x|)\geq0.\eaa$$
We have thereby shown that $v^T$ is a supersolution to~\eqref{homo} outside the ball $\overline{B_{R'}}$. Observe now that~\eqref{v't0} implies that $v^T(t,\.)$ is of class $W^{2,\infty}$ in a neighborhood of $\partial B_{R'}$, for each~$t\in[0,T]$. Next, for $t\in[0,T]$ and $|x|<R'$, using~\eqref{v2super} and~\eqref{v''t0} we get
$$\partial_t v^T(t,x)-\Delta v^T(t,x)=\partial_t v(t,0)>\partial_{rr}v(t,0)+f(v(t,0))>f(v^T(t,x)).$$
Summing up, the function $v^T$ is a positive generalized supersolution~to~\eqref{homo}~in~$[0,T]\!\times\!\R^N$. We further see from the definition of $v$ and from~\eqref{olphi} and~\eqref{vt0} that $v^T$ satisfies
\Fi{vT10'}\left\{\baa{ll}
v^T(0,x)\geq 1, & \forall\,|x|\geq R'+L+c'T,\vspace{3pt}\\
v^T(t,0)<\ol\vp(L), & \forall\,t\in[0,T].\eaa\right.
\Ff

\noindent{\it{Step 3: conclusion.}} Consider any $c>c^*$ and $\lambda>0$. Let any $c'$ and $c''$ be such that $c>c'>c''>c^*$, and let $s_0$, $Z$, $\beta$, $L$ and $R'$ be the positive parameters (depending on~$f$, $N$, $c'$ and~$c''$, whence on $f$, $N$ and $c$, since $c'$ and $c''$ depend on $c$ and $c^*$ while $c^*$ depends on~$f$ only) given as in~\eqref{f>f'2}-\eqref{defR0}. Without loss of generality, one can also assume now that~$L$ is large enough (depending also on $\lambda$) so that $\ol\vp(L)\le\lambda$. Let $(v^T)_{T>0}$ be the functions defined as in Step~2 above. Since $c>c'$, the conclusion~\eqref{vT10} with $R=R'+L>0$ follows from~\eqref{vT10'}. The proof of Proposition~\ref{proradial} is thereby complete.
\end{proof}

The next simple lemma concerns the sets $\mc{B}(U)$ and $\mc{U}(U)$ introduced at the beginning of Section~\ref{sec:FG}.

\begin{lemma}\label{lem:eB}
Let $U$ be a non-empty subset of $\R^N$ satisfying~\eqref{hyp:U} for some $\rho>0$. Then for every~$e\in\mc{B}(U)$, there holds that
\be\label{conceB}
\liminf_{\tau\to+\infty}\frac{\dist(\tau e,U)}{\tau}=\inf_{\xi\in\mc{U}(U),\,\xi\.e\ge0}\,\sqrt{1-(\xi\.e)^2}>0,
\ee
with the convention that the right-hand side of the equality is $1$ if there is no $\xi\in\mc{U}(U)$ satisfying~$\xi\.e\ge0$.
\end{lemma}

\begin{proof}
Call
$$\bar\delta:=\liminf_{\tau\to+\infty}\frac{\dist(\tau e,U)}{\tau},\qquad m:=\inf_{\xi\in\mc{U}(U),\,\xi\.e\ge0}\,\sqrt{1-(\xi\.e)^2},$$
with the convention that $m=1$ if there is no $\xi\in\mc{U}(U)$ satisfying~$\xi\.e\ge0$. The definition of $\mc{B}(U)$ yields $\bar\delta>0$. 

We start with proving $\bar\delta\leq m$. If there exists no $\xi\in\mc{U}(U)$ satisfying $\xi\.e>0$ then~$m=1\geq\bar\delta$ because $U\neq\emptyset$. Suppose now that there is $\xi\in\mc{U}(U)$ such that~$\xi\.e>0$. The definition of $\mc{U}(U)$ yields the existence of a family of points $(x_\tau)_{\tau>0}$ in~$U$ such that
$$\big|\xi-\frac{x_\tau}{\tau}\big|\to0\as\tau\to+\infty.$$
It follows that
$$\displaystyle\bar\delta\leq\lim_{\tau\to+\infty}\frac{\Big|\frac\tau{\xi\.e}e-x_\tau\Big|}{\frac\tau{\xi\.e}}=\lim_{\tau\to+\infty}\big|e-\frac{\xi\.e}\tau x_\tau\big|=|e-(\xi\.e)\xi|=\sqrt{1-(\xi\.e)^2}.$$
Since this holds for any $\xi\in\mc{U}(U)$ such that $\xi\cdot e>0$, the inequality $\bar\delta\leq m$ follows.

Let us pass now to the proof of the reverse inequality $m\le\bar\delta$. Remember first that~$0\le\bar\delta\le1$. If $\bar\delta=1$ it trivially holds because $m\leq1$. Suppose that $\bar\delta<1$. There exist a positive sequence $\seq{\tau}$ diverging to $+\infty$ and a sequence of points $\seq{x}$ in $U$ satisfying
$$\Big|e-\frac{x_n}{\tau_n}\Big|<\bar\delta+\frac1n\ \hbox{ \ for all $n\in\N$}.$$
Because $\bar\delta<1$, we see that $|x_n|\to+\infty$ as $n\to+\infty$, hence we can assume that $x_n\neq0$ for all $n\in\N$. We further deduce from the above inequality that
$$\frac{(x_n\.e)^2}{|x_n|^2}\geq2\frac{x_n\.e}{\tau_n}-\frac{|x_n|^2}{\tau_n^2}=1-\Big|e-\frac{x_n}{\tau_n}\Big|^2>1-\Big(\bar\delta+\frac1n\Big)^2.$$
hence, calling $\xi_n:=x_n/|x_n|$,
$$\xi_n\.e\geq\sqrt{1-\Big(\bar\delta+\frac1n\Big)^2}.$$
Thus, the limit $\ol\xi\in\Sph$ of a converging subsequence of $\seq{\xi}$ satisfies $\ol\xi\.e\geq\sqrt{1-\bar\delta^2}$. Furthermore,  there holds that
$$\frac{\big||x_n|\ol\xi-x_n\big|}{|x_n|}=|\ol\xi-\xi_n|\to0\as n\to+\infty,$$
which means that $\ol\xi\notin\mc{B}(U)$. Hence, by~\eqref{hyp:U}, $\ol\xi\in \mc{U}(U_\rho)\subset\mc{U}(U)$. We eventually derive
$$\sup_{\xi\in\mc{U}(U)}\,\xi\.e\geq\ol\xi\cdot e\geq\sqrt{1-\bar\delta^2}>0,$$
that is, $\bar\delta\geq m$. The proof of Lemma~\ref{lem:eB} is thereby complete.
\end{proof}

Thanks to Lemma~\ref{lem:eB}, and using the family of supersolutions provided by Proposition~\ref{proradial}, one can finally establish a set where the solution converges to $0$.

\begin{lemma}\label{lem:BU}
Assume that Hypothesis~$\ref{hyp}$ holds, let $c^*>0$ be given by Proposition~$\ref{pro1}$, and let $u$ be the solution of~\eqref{homo}-\eqref{defu0} with $U\neq\emptyset$ satisfying~\eqref{hyp:U} for some $\rho>0$. Assume that there exists $e\in\mc{B}(U)$. Then, for any $w>w(e)$, where $w(e)$ is given by~\eqref{FGgeneral}, in the cone
$$\mc{C}:=\bigcup_{\tau>1}B_{c^*(\tau-1)}(\tau w e),$$
there holds that
$$\sup_{x\in \mc{C}}\,u(t,tx)\to0\as t\to+\infty.$$
\end{lemma}

\begin{proof}
Consider $U,u,\rho,e,w$ and $\mc{C}$ as in the statement. Because $w>w(e)>0$, there exists a real number~$k$ satisfying
\be\label{choicek}
0<\frac{c^*}w<k<\frac{c^*}{w(e)}=\inf_{\su{\xi\in\mc{U}(U)}{\xi\.e\ge0}}\sqrt{1-(\xi\.e)^2}.
\ee
Then, by Lemma~\ref{lem:eB}, we can find $\tau_1>0$ such that
\Fi{disteU}
\forall\,\tau\geq\tau_1,\quad{\dist(\tau e,U)}\geq k{\tau}.
\Ff
Take $c\in(c^*,kw)$ and $\lambda>0$. Consider the associated constant $R>0$ and the functions~$(v^T)_{T>0}$ provided by Proposition~\ref{proradial}. By~\eqref{disteU}, for all $T>0$, there holds that
$$\forall\,\tau>\tau_1+\frac{R+cT}{k}, \ \ \forall\,y\in B_{{k}\tau-{R-cT}}(\tau e),\quad\dist(y,U)\geq R+cT,$$
hence, for $\tau$ and $y$ as above, $u$ is less than or equal to the positive function $v^T(\.,\.-y)$ at time $0$, due to~\eqref{vT10}, and therefore the comparison principle yields~$u(T,y)\leq v^T(T,0)<\lambda$. We~rewrite this inequality using $t=T$, $s=\tau/T$, $x=y/T$, that is,
\be\label{tsx}
\forall\,t>0,\ \ \forall\,s>\frac{c}k+\frac1t\Big({\tau_1}+\frac{R}{k}\Big), \ \ \forall\,x\in B_{{k}s-c-R/t}(s e),\quad u(t,tx)<\lambda.
\ee
If we show that, for $t$ sufficiently large, any point $x\in\mc{C}$ can be written in the above form, the lemma is proved, due to the arbitrariness of $\lambda$. Consider $x\in\mc{C}$. We write it as~follows:
\be\label{sxyx}
x=s_xe+y_x,\quad\text{with }\,s_x>w\hbox{ and }|y_x|<c^*\Big(\frac{s_x}{w}-1\Big).
\ee
Then, recalling that $c<kw$, we can find $t_1>0$ (depending on $c,k,w,\tau_1,R$, but not on~$x$), such that 
$$\forall\,t>t_1,\quad s_x>w>\frac{c}k+\frac1t\Big({\tau_1}+\frac{R}{k}\Big).$$
Next, using $c^*/w<k$ and $s_x>w$ in~\eqref{sxyx}, we infer that
$$|y_x|-ks_x<\Big(\frac{c^*}{w}-k\Big)s_x-c^*<\Big(\frac{c^*}{w}-k\Big)w-c^*=-kw.$$
Then, because $c<kw$, we can find $t_2\in[t_1,+\infty)$ (depending on $t_1,c,k,w,R$, but not on~$x$), such that $-kw<-c-R/t$ for all $t>t_2$, hence $|y_x|<ks_x-c-R/t$ for $t>t_2$. We have shown that $x$, $s\!=\!s_x$ fulfill the inclusion and inequality in~\eqref{tsx}, hence the proof is~concluded.
\end{proof}

%-----------------------------------------------------------------------------------
	
\section{Proofs of Theorems~\ref{th1}-\ref{th4}}\label{sec5}

This section is devoted to the proofs of Theorems~\ref{th1}-\ref{th4}. We also state and show  Proposition~\ref{pro:hypU} below on some sufficient conditions for the validity of the geometric  hypothesis~\eqref{hyp:U}. Lastly, Remark~\ref{remtheta} below discusses the case of more general initial conditions than~\eqref{defu0}.

We present the proofs of our main theorems in the following order: \ref{th4},\ref{th2},\ref{th1},\ref{th3}.

\begin{proof}[Proof of Theorem~$\ref{th4}$]
Fix $c\in(0,c^*)$, where $c^*>0$ is given by Proposition~\ref{pro1}. It follows from the assumption~\eqref{dUrho} that, for given~$c'\in(c,c^*)$, the inclusion $U+B_{ct}\subset U_\rho+B_{c't}$ holds for $t>0$ sufficiently large (depending on~$c$ and~$c'$). Thus, Lemma~\ref{lem:U+B} implies that $\inf_{x\in U+B_{ct}}u(t,x)\to1$ as $t\to+\infty$. Therefore, for any $\lambda\in(0,1)$, there holds that
$$U+B_{ct}\subset E_\lambda(t),$$
for $t$ sufficiently large. Since this holds for each $c\in(0,c^*)$, one infers that
\be\label{hausdorff1}
\sup_{x\in U+B_{c^*t}}\dist(x,E_\lambda(t))=o(t)\ \hbox{ as }t\to+\infty.
\ee

Conversely, by taking any $c'>c^*$ and $\lambda\in(0,1)$, we will show that
\be\label{Elambda}
\big\{x\in\R^N:\dist(x,U)\geq c't\big\}\subset \R^N\setminus E_\lambda(t),
\ee
for $t$ sufficiently large. To do so, consider any $c\in(c^*,c')$, and let $R>0$ and $(v^T)_{T>0}$ be given by Proposition~\ref{proradial}. Denote $t_0=R/(c'-c)>0$, and consider any $t\ge t_0$. For any $x_0\in\R^N$ such that $\dist(x_0,U)\ge c't$, one has $B_{c't}(x_0)\subset\R^N\!\setminus\!U$, hence $u_0\le\1_{\R^N\setminus B_{c't}(x_0)}$ and $u_0\le v^t(0,\cdot-x_0)$ in $\R^N$ by~\eqref{vT10} (observe that $c't\ge R+ct$). Since $v^t$ is a generalized supersolution, the maximum principle yields $u(t,\cdot)\le v^t(t,\cdot-x_0)$ in~$\R^N$, hence $u(t,x_0)\le v^t(t,0)<\lambda$ by~\eqref{vT10}, that is, $x_0\in\R^N\setminus E_\lambda(t)$. Therefore, one has shown~\eqref{Elambda}, or equivalently $E_\lambda(t)\subset U+B_{c't}$, for all $t\ge t_0$. Since this holds for each $c'>c^*$, one~deduces
$$\sup_{x\in E_\lambda(t)}\dist(x,U+B_{c^*t})=o(t)\ \hbox{ as }t\to+\infty.$$
Together with~\eqref{hausdorff1}, this gives~\eqref{dH}. 
\end{proof}

\begin{proof}[Proof of Theorem~$\ref{th2}$] 
To start with, we show that the envelop set $\mc{W}$ of the function $w$ defined by~\eqref{FGgeneral}-\eqref{conv} has the geometric expression~\eqref{asspre-formula}. On one hand, if $\mc{U}(U)\neq\emptyset$ then~$w$ is given by the formula~\eqref{FGgeometric} (with the convention $c^*/0=+\infty$). Hence, in such a case, for any $e\in\Sph$ and $r\ge0$, it holds that $\dist(re,\R^+\mc{U}(U))=r\,\dist(e,\R^+\mc{U}(U))=rc^*/w(e)$, which yields the equivalence between~\eqref{asspre} and~\eqref{asspre-formula}. On the other hand, if $\mc{U}(U)=\emptyset$ then $\mc{W}\equiv B_{c^*}$ owing to~\eqref{conv}, and therefore~\eqref{asspre-formula} holds true in this case under our convention $\R^+\emptyset\,+\,B_{c^*}=B_{c^*}$.

We now turn to~\eqref{ass-cpt}. We preliminarily observe that, by~\eqref{hyp:U}, one has
$$\mc{U}(U)\supset \mc{U}(U_\rho)=\Sph\setminus \mc{B}(U)\supset \mc{U}(U),$$
that is, $\mc{U}(U_\rho)=\mc{U}(U)$. Fix a compact set $C$ contained in $\mc{W}$. For any $\xi\in\mc{U}(U)=\mc{U}(U_\rho)$ (if it exists), and any~$\tau>0$ and $0<c'<c<c^*$, the definition of $\mc{U}(U_\rho)$ yields
$$\frac1t\dist(t \tau \xi,U_\rho)\to0\as t\to+\infty,$$
hence $B_{c't}(t \tau \xi)\subset U_\rho+B_{ct}$ for $t$ sufficiently large. It then follows from Lemma~\ref{lem:U+B} that
\Fi{Bctauxi}
\inf_{x\in B_{c'}(\tau \xi)}u(t,tx)\to1\as t\to+\infty.
\Ff
Note that the above limit holds good when $\tau=0$ (without any reference to $\xi$) due to Proposition~\ref{pro1} and the fact that $u_0$ fulfills~\eqref{hyptheta} for some $x_0\in\R^N$, because $U_\rho\neq\emptyset$. Moreover, the expression~\eqref{asspre-formula} implies that any point $x\in \mc{W}$ is contained either in $B_{c'_x}$ or in $B_{c'_x}(\tau_x\xi_x)$, for certain $c'_x\in(0,c^*)$,  $\xi_x\in\mc{U}(U)$, and $\tau_x>0$. Then, by compactness, $C$ can be covered by a finite number of such balls and therefore, since~\eqref{Bctauxi} holds in each one of them, the first limit in~\eqref{ass-cpt} follows.

Consider now a compact set $C$ included in $\R^N\setminus\ol{\mc{W}}$. Any point $y\in C$ is such that~$e:=y/|y|$ satisfies $w(e)<|y|<+\infty$, hence necessarily $e\in\mc{B}(U)$, because otherwise~\eqref{hyp:U} would yield $e\in\mc{U}(U_\rho)=\mc{U}(U)$ and then $w(e)=+\infty$ by the convention~\eqref{conv}. As a consequence, for an arbitrary $\lambda>0$, applying Lemma~\ref{lem:BU} with $w\in(w(e),|y|)$, we infer the existence of an open neighborhood $\mc{C}_y$ of $y$ and of some~$t_y>0$ such~that
$$\forall\,t>t_y,\ \forall x\in\mc{C}_y,\quad u(t,tx)<\lambda.$$
By a covering argument we can find $t_C>0$ such that
$$\forall\,t>t_C,\ \forall x\in C,\quad u(t,tx)<\lambda.$$
This gives the second limit in~\eqref{ass-cpt}, concluding the proof of Theorem~\ref{th2}.
\end{proof}

\begin{proof}[Proof of Theorem~$\ref{th1}$] The first limit in~\eqref{ass} is a particular instance of the first limit in the conclusion~\eqref{ass-cpt} of Theorem~\ref{th2}. The second one only involves the directions $e$ for which $w(e)<+\infty$, whence, by~\eqref{hyp:U}, the ones in $\mc{B}(U)$. 
Thus such a limit immediately follows by applying Lemma~\ref{lem:BU} with $w\in(w(e),c)$.
\end{proof}

\begin{proof}[Proof of Theorem~$\ref{th3}$] 
We derive the limit~\eqref{Hloc} in the more general case where $\ol{B_R}$ is replaced by any compact set $K\subset\R^N$ satisfying $\overline{K\cap\mc{W}}=K\cap\overline{\mc{W}}$ (which holds true when $K=\ol{B_R}$ by the definition of $\mc{W}$).

Take $\lambda\in(0,1)$ and $\e>0$. For $\eta>0$, we define the following subset of $K\cap\mc{W}$:
$$ K_\eta:=K\cap\big\{re\,:\,e\in \Sph,\ \ 0\leq r\leq w(e)-\eta\big\}.$$
From the continuity of $w$, it follows that $K_\eta$ is a compact set included in the open set~$\mc{W}$.  On the one hand, since $\ol{K\cap\mc{W}}$ is compact, using a covering argument one can find $\eta>0$ small enough such that
$$K\cap\mc{W}\subset \ol{K\cap\mc{W}} \subset K_\eta+B_\e.$$ 
On the other hand, by the first line of the conclusion~\eqref{ass-cpt} of Theorem~\ref{th2} applied with $C={K}_{\eta}$, we infer that, for~$t$ larger than some $T>0$ depending on $\eta$, there holds that $K_{\eta}\subset t^{-1}E_\lambda(t)$ and therefore~$K_{\eta}\subset K\cap t^{-1}E_\lambda(t)$. Combining these inclusions one then~gets 
\Fi{K<}
\forall\,t>T,\quad K\cap\mc{W}\subset \Big(K\cap\frac1tE_\lambda(t)\Big) +B_\e.
\Ff

Consider now, for $\sigma>0$, the set
$$K'_\sigma:=K\cap\big\{re\,:\,e\in \Sph,\ \ r\geq w(e)+\sigma\big\}.$$
By the continuity of $w$, this is a compact set contained in $\R^N\setminus\ol{\mc{W}}$. Let us check that
\Fi{K-K'}
K\setminus K'_\sigma\subset \big(K\cap\overline{\mc{W}}\big)+B_\e,
\Ff
for all $\sigma>0$ small enough. Assume by contradiction that this is not the case. Then we can find a sequence $(r_n e_n)_{n\in\N}$ in $K\setminus\big((K\cap\overline{\mc{W}})+B_\e\big)$ with $(e_n)_{n\in\N}$ in $\Sph$ and~$(r_n)_{n\in\N}$ bounded and satisfying $r_n<w(e_n)+1/n$ for all $n\in\N$. Thus, up to subsequences,~$(e_n)_{n\in\N}$ converges to some $e\in\Sph$ and then, by the continuity of $w$, $(r_n)_{n\in\N}$ converges to some~$r\leq w(e)$ (whenever $w(e)$ be finite or not). This means that $r e\in K\cap\overline{\mc{W}}$ and therefore $r_n e_n\in(K\cap\overline{\mc{W}})+B_\e$ for $n$ large, a contradiction. We can then choose $\sigma>0$ such that~\eqref{K-K'} holds. Applying the second line of the conclusion~\eqref{ass-cpt} of Theorem~\ref{th2} with $C={K}_{\sigma}'$, we can find $\tau>0$ such that
$$\forall\,t>\tau,\quad K'_{\sigma}\cap \frac1t E_\lambda(t)=\emptyset,$$
whence
$$\forall\,t>\tau,\quad K\cap\frac1t E_\lambda(t)\subset K\setminus K'_{\sigma}.$$
But then, since the inclusion~\eqref{K-K'} and the fact that $K\cap\overline{\mc{W}}=\ol{K\cap\mc{W}}$ yield $K\setminus K'_{\sigma}\subset \ol{K\cap\mc{W}}+B_\e$, one eventually derives
$$\forall\,t>\tau,\quad K\cap\frac1t E_\lambda(t)\subset \ol{K\cap\mc{W}}+B_\e=(K\cap\mc{W})+ B_\e.$$
This property, together with~\eqref{K<}, yields the desired result~\eqref{Hloc}, owing to the arbitrariness of~$\e>0$.
\end{proof}

The last result of this section provides a list of conditions for a set $U\subset\R^N$ to fulfill property~\eqref{hyp:U}. As we will see in the examples listed in Section~\ref{sec6}, conditions~\eqref{hyp:U} and~\eqref{dUrho} cannot be compared. However, condition~\eqref{dUrho} together with certain additional properties imply~\eqref{hyp:U}, as the following result shows.

\begin{proposition}\label{pro:hypU}
For a set $U\subset\R^N$, property~\eqref{hyp:U} holds if $U$ satisfies~\eqref{dUrho} together with one of the following conditions:
\begin{itemize}	
\item either $U$ is star-shaped with respect to some point $x_0\in\R^N$;
\item or there exists $U'\subset\R^N$ satisfying
\Fi{hyp:U'}
\mc{B}(U')\cup\,\mc{U}(U')=\Sph
\Ff
and $d_{\mc{H}}(U,U')<+\infty$;
\item or there exists $U'\subset\R^N$ satisfying~\eqref{hyp:U'} and
\Fi{dRUU'}
\frac{d_{\mc{H}}(U\cap B_R,U'\cap B_R)}{R}\longrightarrow 0\as R\to+\infty.
\Ff
\end{itemize}
\end{proposition}

\begin{proof}
First of all, using~\eqref{dUrho} one sees that, for any $\xi\in\mc{U}(U)$,
$$\frac{\dist(\tau\xi,U_\rho)}{\tau}\leq\frac{\dist(\tau\xi,U)+d_{\mc{H}}(U,U_\rho)}{\tau}\to0\as \tau\to+\infty,$$
that is, $\xi\in \mc{U}(U_\rho)$. Thus, it is sufficient to show~\eqref{hyp:U}  with $\mc{U}(U)$ instead of~$\mc{U}(U_\rho)$. 

Consider the case where $U$ is star-shaped. Since properties~\eqref{dUrho} and~\eqref{hyp:U} are invariant under rigid transformations of the coordinate system, we can assume without loss of generality that $U$ is star-shaped with respect to the origin. Suppose that there exists~$\xi\in\Sph\setminus \mc{B}(U)$ (otherwise property~\eqref{hyp:U} trivially holds). This means that there exists a sequence $\seq{\tau}$ diverging to $+\infty$ and a sequence $\seq{x}$ in $U$ such that~$|\tau_n\xi-x_n|/\tau_n\to0$ as $n\to+\infty$. Then, for any $0<\tau\leq\tau_n$, since $ \frac{\tau}{\tau_n}x_n\in U$ because $U$ is star-shaped with respect to the origin, one finds
$$\frac{\dist(\tau\xi,U)}{\tau}\leq\frac{|\tau\xi-\frac{\tau}{\tau_n}x_n|}\tau=\Big|\xi-\frac{1}{\tau_n}x_n\Big|\to0\as n\to+\infty,$$
whence $\xi\in \mc{U}(U)$. This shows that $\mc{B}(U)\cup\mc{U}(U)=\Sph$ and, as already emphasized, this proves the statement in this case.

Consider now the hypotheses of the second case, with $U'\subset\R^N$ satisfying~\eqref{hyp:U'} and~$d_{\mc{H}}(U,U')<+\infty$. Then there holds that 
\Fi{UU'}
\mc{B}(U')=\mc{B}(U)\quad\text{ and }\quad\mc{U}(U')=\mc{U}(U),
\Ff
hence $\mc{B}(U)\cup\mc{U}(U)=\Sph$ and, as above,~\eqref{hyp:U} then follows.

Let us check that the same conclusions~\eqref{UU'} hold when $U'$ satisfies~\eqref{hyp:U'}-\eqref{dRUU'}. We~call
$$D_R:=d_{\mc{H}}(U\cap B_R,U'\cap B_R).$$
For $\xi\in\Sph$ and $\tau>0$, there exists $x_\tau\in U$ such 
$$|\tau\xi-x_\tau|<\dist(\tau\xi,U)+1,$$
then in particular 
\Fi{xtau<}
|x_\tau|<\tau+\dist(\tau\xi,U)+1\le\tau+|\tau\xi-x_1|+1\le2\tau+|x_1|+1.
\Ff
Moreover, we can find $x_\tau'\in U'\cap B_{|x_\tau|+1}$ for which $|x_\tau'-x_\tau|< D_{|x_\tau|+1}+1$. It follows that
$$\dist(\tau\xi,U')\leq|\tau\xi-x_\tau'|\leq\dist(\tau\xi,U)+1+D_{|x_\tau|+1}+1.$$
By~\eqref{dRUU'} and~\eqref{xtau<} one then deduces the inequality
$$\dist(\tau\xi,U')\leq\dist(\tau\xi,U)+o(\tau)\as\tau\to+\infty,$$
and then $|\dist(\tau\xi,U')-\dist(\tau\xi,U)|=o(\tau)$ as $\tau\to+\infty$ by switching the roles of $U$ and~$U'$. From this, the equivalences~\eqref{UU'} immediately follow, and one concludes as in the previous paragraph.
\end{proof}

\begin{remark}\label{remtheta}{\rm The conclusions of Theorems~$\ref{th1}$-$\ref{th4}$ still hold for the solutions to~\eqref{homo} with initial conditions more general than characteristic functions in~\eqref{defu0}. To be more precise, firstly, if Hypothesis~$\ref{hyp}$ is satisfied, if the minimal speed $c^*$ and the parameters $\theta\in(0,1)$ and $\rho>0$ are given by Proposition~$\ref{pro1}$, and if $u$ is a solution to~\eqref{homo} such that
$$\{u_0\ge\theta\}_\rho\neq\emptyset,\ \ \mc{B}\big(\{u_0\ge\theta\}\big)\cup\,\mc{U}\big(\{u_0\ge\theta\}_\rho\big)=\Sph$$
and
$$d_{\mc{H}}\big(\supp u_0,\{u_0\ge\theta\}\big)<+\infty,$$
then the conclusions of Theorem~$\ref{th1}$-$\ref{th3}$ hold, with $\mc{U}(U)$ replaced by~$\mc{U}\big(\{u_0\ge\theta\}\big)$ in the definitions~\eqref{FGgeneral}-\eqref{conv} of $w(e)$. Indeed, it is easy to see that $\mc{B}(\supp u_0)=\mc{B}\big(\{u_0\ge\theta\}\big)$, that $\mc{U}(\supp u_0)=\mc{U}\big(\{u_0\ge\theta\}\big)$, that Lemma~$\ref{lem:U+B}$ holds with~$U$ replaced by $\{u_0\ge\theta\}$, and that Lemmas~$\ref{lem:eB}$ and~$\ref{lem:BU}$ hold as well with $\mc{B}(U)$ replaced by~$\mc{B}(\supp u_0)$ in both statements and~$U$ replaced by $\supp u_0$ in~\eqref{conceB}. Meanwhile, Proposition~$\ref{proradial}$ is kept unchanged.

Secondly, if Hypothesis~$\ref{hyp}$ is satisfied, if the minimal speed $c^*$ and the parameters $\theta\in(0,1)$ and $\rho>0$ are given by Proposition~$\ref{pro1}$, and if~$u$ is a solution to~\eqref{homo} such that
$$\{u_0\ge\theta\}_\rho\neq\emptyset,\ \ d_{\mc{H}}\big(\supp u_0,\{u_0\ge\theta\}_\rho\big)<+\infty$$
and
$$d_{\mc{H}}\big(\supp u_0,\{u_0\ge\theta\}\big)<+\infty,$$
then the conclusion of Theorem~$\ref{th4}$ is changed into the following property: for any $\lambda\in(0,1)$,
$$d_{\mc{H}}\big( E_\lambda(t)\,,\,\supp u_0+B_{c^*t}\big)=o(t)\ \hbox{ and }\ d_{\mc{H}}\big( E_\lambda(t)\,,\,\{u_0\ge\theta\}+B_{c^*t}\big)=o(t)$$
as $t\to+\infty$. Indeed, $\mc{B}(\supp u_0)=\mc{B}\big(\{u_0\ge\theta\}\big)$, $\mc{U}(\supp u_0)=\mc{U}\big(\{u_0\ge\theta\}\big)$, and Lemma~$\ref{lem:U+B}$ holds with~$U$ replaced by $\{u_0\ge\theta\}$, while Proposition~$\ref{proradial}$ is kept unchanged.}
\end{remark}

%-----------------------------------------------------------------------------------
%-----------------------------------------------------------------------------------

\section{Counter-examples}\label{sec6}

In this section, we show some counter-examples to Theorems~\ref{th1}-\ref{th4} and to the formula~\eqref{Hausdorff} when the assumptions~\eqref{hyp:U} or~\eqref{dUrho} are not satisfied. In all the counter-examples, we consider the function $f(s)=s(1-s)$ for $s\in[0,1]$. Hence, the invasion property stated in Definition~\ref{hyp:invasion} holds for arbitrary $\theta\in(0,1)$ and $\rho>0$, as well as Hypothesis~\ref{hyp}, and moreover the minimal speed of planar traveling fronts connec\-ting~$1$ to~$0$ is equal to $c^*=2$, see~\cite{AW,F,KPP}.

\begin{proposition}\label{pro:ce1}
Let $u$ be the solution to~\eqref{homo} with $f(s)=s(1-s)$ and initial datum $u_0=\1_U$, where
$$U=\bigcup_{n\in\N}\overline{B_{2^n+1}}\setminus B_{2^n-1}.$$
The set $U$ has a non-empty interior, it does not fulfill~\eqref{hyp:U} for any $\rho>0$, but it fulfills~\eqref{dUrho} for any~$\rho\le1$ $($hence,~\eqref{dH} holds$)$. Moreover, \eqref{ass0},~\eqref{ass},~\eqref{ass-cpt} and~\eqref{Hloc} all fail, for any function $w:\Sph\to[0,+\infty]$ and any open set~$\mc{W}\subset\R^N$ which is star-shaped with respect to the origin, and both limits in~\eqref{Hausdorff} do not exist.
\end{proposition}

\begin{proof}
On the one hand, the intersection of $U$ with any ray $\R^+e$, $e\in\Sph$, is unbounded, whence $\mc{B}(U)=\emptyset$. On the other hand, for any $e\in\Sph$, the formula
\Fi{32n}
\dist(3\times2^ne,U)=2^n-1
\Ff
shows that $\mc{U}(U)=\emptyset$ too. Therefore~\eqref{hyp:U} is not satisfied. Let us check that formula~\eqref{ass0} (whence the stronger one~\eqref{ass}) does not hold in any given direction $e\!\in\!\Sph$, with any $w(e)\!\in\![0,+\infty]$. Indeed, on the one hand, by Lemma~\ref{lem:U+B}
\Fi{u2n}
\lim_{t\to+\infty}u(t,2^ne)=1\quad\text{uniformly with respect to~$n$}.
\Ff
Thus, if~\eqref{ass0} were satisfied for some $e\in\Sph$, one would necessarily have $w(e)=+\infty$. On the other hand, given~$\lambda\in(0,1)$ and $c=2c^*$, consider the family of functions~$(v^T)_{T>0}$ and the associated $R>0$ provided by Proposition~\ref{proradial}. For any $n\in\N$ satisfying $n>\log_2(R+1)$, we call $T_n:=(2^n-1-R)/(2c^*)>0$ and deduce from the first property in~\eqref{vT10} that 
$$\forall\,|x|\geq 2^n-1,\quad v^{T_n}(0,x)\geq 1,$$
and therefore, because of~\eqref{32n}, $v^{T_n}(0,\.)\geq u_0(\.+3\times2^ne)$ in $\R^N$, for every $e\in\Sph$. Thus, the comparison principle together with the second property in~\eqref{vT10} entail
$$\forall\,t\leq T_n,\quad u(t,3\times2^ne)\leq v^{T_n}(t,0)<\lambda<1,$$
for every $e\in\Sph$. Calling $\tau_n:=2^{n-2}/c^*$, we have that $\tau_n<T_n$ for $n$ large enough, hence we get
\Fi{uTn}
\limsup_{n\to+\infty}\,u(\tau_n,12\, c^*\tau_n e)\leq\lambda<1,
\Ff
for every $e\in\Sph$. Consequently, if~\eqref{ass0} were satisfied for some $e\in\Sph$, one would necessarily~have $w(e)\leq 12\, c^*$, a contradiction with~$w(e)=+\infty$. In conclusion, formula~\eqref{ass0} and then formula~\eqref{ass} do not hold in any direction $e\in\Sph$, for any $w(e)\in[0,+\infty]$. 
	
The set $\mc{W}$ given by~\eqref{asspre} with $w(e)$ as in~\eqref{FGgeneral}-\eqref{conv} is actually equal to $B_{c^*}$. We will see that~\eqref{ass-cpt} and~\eqref{Hloc} fail with $\mc{W}=B_{c^*}$, as well as with any open set $\mc{W}$ which is star-shaped with respect to the origin. So assume now by way of contradiction that there exists an open set $\mc{W}\subset\R^N$ which is star-shaped with respect tot the origin and for which either~\eqref{ass-cpt} or~\eqref{Hloc} hold. Because of~\eqref{uTn}, the first condition in~\eqref{ass-cpt} in one case, or~\eqref{Hloc} in the other case, imply that $\{12\, c^*e\}\notin\mc{W}$, for any $e\in\Sph$. Hence, being star-shaped, $\mc{W}$ satisfies $\mc{W}\subset B_{12\,c^*}$. But we also know that, by~\eqref{u2n}, $u(2^n/\sigma,2^ne)\to1$ as~$n\to+\infty$ for any $\sigma>0$. Taking $\sigma>12\,c^*$, the second line of~\eqref{ass-cpt} is violated by~$C=\{\sigma e\}$ and moreover, for given $\lambda\in(0,1)$, $d_{\mc{H}}(\{\sigma e\}\cap\frac1t\, E_\lambda(t) \,,\, \{\sigma e\}\cap \mc{W})=+\infty$ for~$t=2^n/\sigma$ and~$n$ large enough (depending on $\lambda$), that is,~\eqref{Hloc} fails too. We have reached a contradiction in both cases.
	
Finally, for $n\in\N$ and $e\in\Sph$, calling $t_n:=2^{n-1}/c^*$, we rewrite~\eqref{32n} as
$$\dist(6c^*e\,t_n,U)=2c^*\,t_n-1.$$
We deduce that
$$B_{c^*-1/t_n}(6c^*e)\subset\frac{1}{t_n}\,\Big\{x\in\R^N:\dist(x,U)>c^*\,t_n\Big\}$$
and therefore if $d_{\mc{H}}(t^{-1}U+B_{c^*},\mc{W}')\to0$ as $t\to+\infty$, for some set $\mc{W}'$, then necessarily $6c^*e\notin \ol{\mc{W}'}$. But we see that, for $s_n:=2^n/(6c^*)$, there holds that
$$6c^*e\in\frac{1}{s_n}U,$$
and thus $d_{\mc{H}}(t^{-1}U+B_{c^*},\mc{W}')\to0$ as $t\to+\infty$ would imply $6c^*e\in \ol{\mc{W}'}$. This shows that the second limit in~\eqref{Hausdorff} does not exist, hence the first limit does not exist either, thanks to~\eqref{dH} (notice also that~\eqref{dUrho} is satisfied for any $\rho\in(0,1]$, hence~\eqref{dH} holds thanks to Theorem~\ref{th4}).
\end{proof}

The second counter-example is the counterpart of Proposition~\ref{pro:ce1}, with a set $U$ fulfilling~\eqref{hyp:U} but not~\eqref{dUrho}.

\begin{proposition}\label{pro:ce1bis}
Let $u$ be the solution to~\eqref{homo} with $f(s)=s(1-s)$ and initial datum $u_0=\1_U$, where $U=U_1\cup U_2$ and
\be\label{defU12}\left\{\baa{l}
U_1:=\big\{x\in\R^N:x_1\ge0\hbox{ and }x_2^2+\cdots+x_N^2\le1\big\},\vspace{3pt}\\
U_2:=\big\{x\in\R^N:x_1\ge0\hbox{ and }(x_2-\sqrt{x_1})^2+x_3^2+\cdots+x_N^2\le e^{-x_1^2}\big\}.\eaa\right.
\ee
The set $U$ has a non-empty interior, it does not fulfill~\eqref{dUrho} for any $\rho>0$, but it fulfills~\eqref{hyp:U} for $0<\rho\le1$ $($hence,~\eqref{ass0},~\eqref{ass},~\eqref{ass-cpt},~\eqref{Hloc} hold$)$. Moreover,~\eqref{dH} fails and the first limit in~\eqref{Hausdorff} exists and is equal to~$\mc{W}$, whereas the second one does not exist.
\end{proposition}

\begin{proof}
First of all, it is immediate to see that $\mc{U}(U)=\mc{U}(U_\rho)=\{\mathrm{e}_1\}$ for any $0<\rho\le1$, with $\mathrm{e}_1=(1,0,\cdots,0)$, while $\mc{B}(U)=\Sph\setminus\{\mathrm{e}_1\}$. This means that the assumption~\eqref{hyp:U} is fulfilled for any $0<\rho\le1$. One readily checks that, instead,~\eqref{dUrho} fails, for any $\rho>0$. The set~$\mc{W}$ given by the equivalent formulas~\eqref{asspre} and~\eqref{asspre-formula} is equal to the rounded half-cylinder $\mc{W}=\R^+\{\mathrm{e}_1\}+B_{c^*}$. It is not hard to see that the second limit in~\eqref{Hausdorff} does not exist, owing to the presence of $U_2$ in the definition of $U$.

It turns out that the presence of $U_2$ in the definition of $U$ does not affect the asymptotic of $E_\lambda(t)$ as $t\to+\infty$. To see this we observe that, since the function $f$ vanishes at $0$ and~$1$ and is concave, the maximum principle yields
\be\label{ineqv12}
0\le\max(v_1,v_2)\le u\le\min(v_1+v_2,1)\ \hbox{ in }[0,+\infty)\times\R^N,
\ee
where $v_i$ solves~\eqref{homo} with initial condition $v_i(0,\cdot)=\1_{U_i}$, for $i=1,2$. Let us call~$E^i_\lambda(t):=\{x\in\R^N:v_i(t,x)>\lambda\}$. Using the comparison with the linearized equation $\partial_tw=\Delta w+w$ and the explicit solution for the latter, one can check that~$v_2(1,x)$ has a Gaussian decay for $|x|\to+\infty$. It then follows from the standard theory that properties~\eqref{c<c*}-\eqref{c>c*} hold for $v_2$, hence $d_{\mc{H}}(t^{-1}E^2_\lambda(t),B_{c^*})\to0$ as $t\to+\infty$ for any~$\lambda\in(0,1)$. On the other hand, the set $U_1$ given in~\eqref{defU12} fulfills both~\eqref{hyp:U} and~\eqref{dUrho} with $0<\rho\le1$, hence the conclusions of Theorems~\ref{th1}-\ref{th4} hold for $v_1$. In particular, $t^{-1}d_{\mc{H}}(E^1_\lambda(t),U_1+B_{c^*t})\to0$ as $t\to+\infty$ for any $\lambda\in(0,1)$. Together with~\eqref{ineqv12}, one infers that, for any $\lambda\in(0,1)$,
\be\label{U1}
\frac1t\,d_{\mc{H}}(E_\lambda(t),U_1+B_{c^*t})\to0\ \hbox{ as }t\to+\infty.
\ee
As a consequence, $d_{\mc{H}}(t^{-1}E_\lambda(t),\mc{W})\to0$ as $t\to+\infty$ for any $\lambda\in(0,1)$ (that is, the first limit in~\eqref{Hausdorff} exists and is equal to $\mc{W}$). But since~$d_{\mc{H}}(U+B_{c^*t},U_1+B_{c^*t})=+\infty$ for all~$t>0$,~\eqref{U1} also implies that, for any fixed $\lambda\in(0,1)$, $d_{\mc{H}}(E_\lambda(t),U+B_{c^*t})=+\infty$ for all~$t$ large enough, hence~\eqref{dH} fails.
\end{proof}

We now exhibit an example where all the conclusions of Theorems~\ref{th1}-\ref{th4} fail and moreover the two limits in~\eqref{Hausdorff} exist but they do not coincide.

\begin{proposition}\label{pro:ce2} 
Let $u$ be the solution to~\eqref{homo} with $f(s)=u(1-s)$ and initial datum $u_0=\1_U$, where
$$U=\big\{x\in\R^N:|x_N|\le e^{-|x'|^2}\big\}.$$
The set $U$ has a non-empty interior, and it does not fulfill~\eqref{hyp:U} or~\eqref{dUrho}, for any $\rho>0$. Then~\eqref{ass0},~\eqref{ass},~\eqref{ass-cpt},~\eqref{Hloc} and~\eqref{dH} all fail with $w(e)$ and $\mc{W}$ given by~\eqref{FGgeneral}-\eqref{conv} and~\eqref{asspre}, and the two limits in~\eqref{Hausdorff} exist but do not coincide.
\end{proposition}

\begin{proof}
We have that $\mc{B}(U)=\{e\in\Sph:e_N\neq0\}$ and that $\mc{U}(U)=\{e\in\Sph:e_N=0\}$ and $\mc{U}(U_\rho)=\emptyset$ for any $\rho>0$. Hence $\mc{B}(U)\cup\mc{U}(U_\rho)\neq\Sph$. The set $\mc{W}$ defined in the equivalent formulations~\eqref{asspre} and~\eqref{asspre-formula} is given by the slab
$$\mc{W}=\big\{x\in\R^N:|x_N|<c^*\big\},$$
and it is readily seen that 
$$d_{\mc{H}}(t^{-1}U+B_{c^*},\mc{W})\to0\ \hbox{ as }t\to+\infty.$$
However, as for the function $v_2$ in the proof of Proposition~\ref{pro:ce1bis}, one has
\be\label{spreadingc*}
d_{\mc{H}}(t^{-1}E_\lambda(t),B_{c^*})\to0\ \hbox{ as }t\to+\infty,
\ee
for any $\lambda\in(0,1)$. Namely, the first limit in~\eqref{Hausdorff} exists and coincides with $B_{c^*}$, and then it is not equal to $\mc{W}$ (in the sense of the Hausdorff distance). We further deduce from~\eqref{spreadingc*} that~\eqref{ass-cpt} and~\eqref{Hloc} fail (just taking $C=K=\ol{B_c}\cap\{x\in\R^N:x_N=0\}$ with $c>c^*$), as well as~\eqref{dH}, because~$d_{\mc{H}}(B_{c^*},t^{-1}U+B_{c^*})=+\infty$ for all $t>0$. Lastly,~\eqref{spreadingc*} implies that the first lines of~\eqref{ass0} and~\eqref{ass} do not hold (because $w(e)$ given by~\eqref{FGgeneral} satisfies $w(e)=+\infty$ for any $e\in\Sph$ with $e_N=0$).
\end{proof}

\begin{remark}\label{remhypU}{\rm
The example given in Proposition~$\ref{pro:ce2}$ further reveals that condition~\eqref{hyp:U} cannot be relaxed by replacing~$\mc{U}(U_\rho)$ with $\mc{U}(U)$, and moreover that, without~\eqref{hyp:U},  formulas~\eqref{ass0} and~\eqref{ass} can hold with some $w(e)$ which is not given by~\eqref{FGgeneral}.}
\end{remark}

We conclude the paper by showing that~\eqref{Hausdorff} may fail even when the hypotheses of Theorems~\ref{th1}-\ref{th4} are fulfilled.

\begin{proposition}\label{pro:ce3} 
Let $u$ be the solution to~\eqref{homo} with $f(s)=s(1-s)$ and initial datum $u_0=\1_U$, where
$$U=\big\{x\in\R^N:x_N\le \sqrt{|x'|}\big\},$$
which has a non-empty interior and fulfills both~\eqref{hyp:U} and~\eqref{dUrho} for any $\rho>0$ $($hence~\eqref{ass0},~\eqref{ass},~\eqref{ass-cpt},~\eqref{Hloc} and~\eqref{dH} all hold$)$. Then~none of the limits in~\eqref{Hausdorff} exist and moreover 
$$\forall\,\lambda\in(0,1),\ \forall\,t>0,\quad d_{\mc{H}}\Big(\,\frac1t\, E_\lambda(t)\,,\,\mc{W}\Big)=+\infty.$$
\end{proposition}

\begin{proof}
It is immediate to see that $\mc{U}(U)=\mc{U}(U_\rho)=\{e\in\Sph:e_N\le0\}$ for any $\rho>0$, and that $\mc{B}(U)=\{e\in\Sph:e_N>0\}$, whence~\eqref{hyp:U} holds. It is also clear that~\eqref{dUrho} holds. We see from~\eqref{asspre-formula} that $\mc{W}=\{x\in\R^N:x_N<c^*\}$. 
		
Next, the functions defined for $n\in\N$ by
$$(t,x)\mapsto u\Big(t,x+n\mathrm{e}_1+\frac{\sqrt n} 2\mathrm{e}_N\Big)$$
converge, as $n\to+\infty$, to the constant solution $\t u(t,x)\equiv1$, locally uniformly in $t\geq0$, $x\in\R^N$. This shows that $d_{\mc{H}}(t^{-1}E_\lambda(t),\mc{W})=+\infty$, for any $\lambda\in(0,1)$ and $t>0$. We deduce that the limit $\lim_{t\to+\infty}t^{-1}E_\lambda(t)$ does not exist, because if it does, it must coincide with~$\mc{W}$ (in the sense of the Hausdorff distance) due to~\eqref{Hloc}. Then the limit $\lim_{t\to+\infty}t^{-1}U+B_{c^*}$ does not exist either, owing to~\eqref{dH}.
\end{proof}

%-----------------------------------------------------------------------------------
%-----------------------------------------------------------------------------------

\end{document}